\documentclass[11pt,a4paper,twoside]{amsart}

\usepackage{amsfonts}
\usepackage{amsmath}
\usepackage{amssymb}
\usepackage{latexsym}
\usepackage[all,poly]{xy}
\usepackage{ifthen}
\usepackage{graphics}
\usepackage{epic}

\newcommand{\confer}{{\em cf.}\ }

\newcommand{\cc}{{\mathcal C}}
\newcommand{\cd}{{\mathcal D}}

\newcommand{\cf}{{\mathcal F}}

\newcommand{\cp}{{\mathcal P}}
\newcommand{\cq}{{\mathcal Q}}

\newcommand{\ct}{{\mathcal T}}

\newcommand{\opname}[1]{\operatorname{\mathsf{#1}}}

\renewcommand{\mod}{\opname{mod}\nolimits}

\newcommand{\add}{\opname{add}\nolimits}

\newcommand{\pr}{\opname{pr}\nolimits}

\newcommand{\rep}{\opname{rep}\nolimits}

\newcommand{\cok}{\opname{cok}\nolimits}

\renewcommand{\ker}{\opname{ker}\nolimits}

\newcommand{\Hom}{\opname{Hom}}

\newcommand{\Ext}{\opname{Ext}}

\newcommand{\End}{\opname{End}}

\newcommand{\LL}{{\ell\ell}}

\newtheorem{Thm}{Theorem}[section]

\newtheorem{theorem}[Thm]{Theorem}
\newtheorem*{theorem*}{Theorem}

\newtheorem{lemma}[Thm]{Lemma}
\newtheorem{proposition}[Thm]{Proposition}
\newtheorem{corollary}[Thm]{Corollary}

\newtheorem*{conjecture*}{Conjecture}

\newtheorem{example}[Thm]{Example}

\newtheorem{definition}[Thm]{Definition}

\newcommand{\ie}{{\em i.e.}}

\title{Endomorphism algebras of maximal rigid objects in cluster tubes}

\author{Dong Yang}
\address{
Dong Yang\\Hausdorff Research Institute for Mathematics\\
Poppelsdorfer Allee 45\\ D-53115 Bonn\\ Germany}
\email{dongyang2002@gmail.com}

\date{ 25 January, 2010. Last modified on \today.}

\begin{document}

\begin{abstract} Given a maximal rigid object $T$ of the cluster tube, we determine the objects finitely presented by $T$.
We then use the method of Keller and Reiten to show that the
endomorphism algebra of $T$ is Gorenstein and of finite
representation type, as first shown by Vatne. This algebra turns out
to be the Jacobian algebra of a certain quiver with potential, when
the characteristic of the base field is not 3. We study how this
quiver with potential changes when $T$ is mutated. We also provide a
derived equivalence
classification for the endomorphism algebras of maximal rigid objects.
\end{abstract}

\maketitle

\tableofcontents
\section{Introduction}\label{S:introduction}

In the theory of additive categorification of Fomin--Zelevsinky's
cluster algebras~\cite{FominZelevinsky02}, 2-Calabi--Yau
triangulated categories with cluster-tilting objects play a central
role, see~\cite{Keller08c} for a nice survey of this topic. A
cluster-tilting object is always a maximal rigid object, while the
converse is generally not true. There exist 2-Calabi--Yau
triangulated categories in which maximal rigid objects are not
cluster-tilting. The first examples of such categories were given by
Burban--Iyama--Keller--Reiten in~\cite{BurbanIyamaKellerReiten08}.

Cluster tubes, introduced by Barot--Kussin--Lenzing
in~\cite{BarotKussinLenzing08}, are another family of 2-Calabi--Yau
triangulated categories without cluster-tilting objects.
In~\cite{BuanMarshVatne10}, Buan--Marsh--Vatne classified maximal
rigid objects of cluster tubes, none of which is cluster-tilting.
In~\cite{Vatne11}, Vatne studied endomorphism algebras of maximal
rigid objects. He gave an explicit description of these algebras in
terms of quivers with relations, and showed that they are Gorenstein
of Gorenstein dimension at most 1 (resembling 2-Calabi--Yau tilted
algebras, \confer~\cite{KellerReiten07}) and are of finite
representation type. One motivation to study these endomorphism
algebras is to categorify cluster algebras of type $B/C$ by cluster
tubes, see~\cite{BuanMarshVatne10,ZhouZhu10b}.

In Section~\ref{s:finitely-presented-objects} of this paper, we give
a categorical explanation of the Gorenstein property and the
representation-finiteness. This is based on a result analogous to
the result of Keller--Reiten~\cite{KellerReiten07} (\confer also
Koenig--Zhu~\cite{KoenigZhu08}): for a maximal rigid object $T$, the
functor $\Hom(T,?)$ induces an equivalence between the module
category of the endomorphism algebra of $T$ and the additive
quotient of a suitable subcategory of the cluster tube by the ideal
generated by the shift of $T$. We determine this suitable
subcategory (Proposition~\ref{p:rigid-is-presented} and
Proposition~\ref{p:finitely-presented-objects}), and the Gorenstein
property and the representation-finiteness follow as consequences.
The Gorenstein property has recently been proved by Zhou--Zhu for
endomorphism algebras of maximal rigid objects in any 2-Calabi--Yau
triangulated categories~\cite{ZhouZhu10}.

In~\cite{BuanMarshVatne10}, maximal rigid objects of a cluster tube
were shown to form a cluster structure with loops. A nice feature of
the cluster structure is the existence of \emph{mutation}. It is of
interest to know the relation between the endomorphism algebras of
two maximal rigid objects related by a mutation. When the mutation
is simple, the two algebras are nearly Morita equivalent in the
sense of Ringel~\cite{Ringel07} ---- this follows from a more
general result (Corollary~\ref{c:nearly-morita-equivalence}); while
when the mutation is not simple, it is not clear whether we can
formulate an analogous statement. In
Section~\ref{s:derived-equivalence}, we study when the two
neighbouring endomorphism algebras are derived equivalent. More
generally, we prove that the endomorphism algebras of two maximal
rigid objects (not necessarily related by a mutation) are derived
equivalent if and only if their quivers have the same number of
3-cycles (Theorem~\ref{t:derived-equivalence}). This derived
equivalence classification is analogous to that of
Buan--Vatne~\cite{BuanVatne08} for cluster-tilted algebras of type
$A$.

In Section~\ref{s:quiver-with-potential}, we associate to each
maximal rigid object a quiver with potential, whose Jacobian algebra
is isomorphic to the endomorphism algebra of the maximal rigid
object. This is a consequence of Vatne's explicit description of the
endomorphism algebra. We study the change of the associated quivers
with potential induced from the mutation of maximal rigid objects
(Proposition~\ref{p:change-of-qp}). In particular, when two maximal
rigid objects are related by a simple mutation, the two associated
quivers with potential are related by the
Derksen--Weyman--Zelevinsky mutation.

Necessary knowledge on cluster tubes, including the definition, the
classification of maximal rigid objects, and the description of
their endomorphism algebras, will be recalled in
Section~\ref{s:preliminary}.
Section~\ref{s:mutations-of-maximal-rigid-objects} is devoted to the
study of mutations of maximal rigid objects of cluster tubes.

Throughout, we fix a field $k$. All vector spaces, algebras,
representations, modules, and categories will be over the field $k$.
We identify a representation of a quiver with the corresponding
right module over the path algebra of the opposite quiver.

\section*{Acknowledgments} The author gratefully acknowledges financial support
from Max-Planck-Institut f\"ur Mathematik in Bonn. He thanks
Bernhard Keller and Pierre-Guy Plamondon for valuable conversations,
and he thanks Bernhard Keller, Pierre-Guy Plamondon, Yann Palu, Yu
Zhou, Bin Zhu and a referee for many helpful remarks on a
preliminary version of this paper.

\section{Preliminaries on the cluster tube}\label{s:preliminary}

Let $n$ be a positive integer.

\subsection{The tube}
Let $\overrightarrow{\Delta}_n$ be the cyclic quiver with $n$
vertices such that arrows are going from $i$ to $i-1$ (taken modulo
$n$).

The \emph{tube of rank $n$} is the category of finite-dimensional
nilpotent representations of the cyclic quiver
$\overrightarrow{\Delta}_n$. It is a hereditary abelian category.
Every indecomposable representation is uniserial, \ie ~it has a
unique composition series, and hence is determined by its socle and
its length up to isomorphism. For $a=1,\ldots,n$ and
$b\in\mathbb{N}$, we will denote by $(a,b)$ the unique (up to
isomorphism) representation with socle the simple at the vertex $a$
and of length $b$. When the first argument does not belong to the
set $\{1,\ldots,n\}$, it should be read as modulo $n$.

The abelian category $\ct_n$ has Auslander--Reiten sequences, and
the Auslander--Reiten translation $\tau$ is an autoequivalence of
$\ct_n$ which takes the indecomposable representation $(a,b)$ to the
indecomposable representation $(a-1,b)$.

For an indecomposable representation $(a,b)$ and an arbitrary
representation $M$ of $\overrightarrow{\Delta}_n$, an extension of
$M$ by $(a,b)$ factors through the canonical inclusion
$(a,b)\hookrightarrow(a,b+l)$ for all $l\geq 1$.

The \emph{Loewy length} of an object $M$ of $\ct_n$, denoted by
$\LL(M)$, is defined as the maximum of the lengths of indecomposable
direct summands of $M$.

\subsection{The cluster tube}
Let $\cd^b(\ct_n)$ be the bounded derived category of the abelian
category $\ct_n$. It is triangulated with suspension functor
$\Sigma$, the shift of complexes. It has Auslander--Reiten
triangles, and the Auslander--Reiten translation is the derived
functor of the Auslander--Reiten translation $\tau$ of $\ct_n$. By
abuse of notation, we also denote it by $\tau$.

\begin{definition}[Barot--Kussin--Lenzing~\cite{BarotKussinLenzing08}]
The \emph{cluster tube of rank $n$} is defined as the orbit category
\[\cc_n:=\cd^b(\ct_n)/\tau^{-1}\circ\Sigma ~~.\]
Precisely, the objects of $\cc_n$ are the same as those of
$\cd^b(\ct_n)$, and for two objects $M$ and $N$ the morphism space
is
\[\Hom_{\cc_n}(M,N)=\bigoplus_{i\in\mathbb{Z}}\Hom_{\cd^b(\ct_n)}(M,(\tau^{-1}\circ\Sigma)^i N).\]
\end{definition}

The category $\cc_n$ has a triangulated structure such that the
canonical projection functor $\pi:\cd^b(\ct_n)\rightarrow \cc_n$ is
triangulated, by Keller~\cite[Theorem 9.9]{Keller05}. It is
2-Calabi--Yau, \ie~ there is a bifunctorial isomorphism
$D\Hom_{\cc_n}(M,N)\cong\Hom_{\cc_n}(N,\Sigma^2 M)$ for objects $M$
and $N$ in $\cc_n$, where $D=\Hom_k(?,k)$ is the $k$-dual. It has
Auslander--Reiten triangles, and the Auslander--Reiten translation
$\tau$ is naturally equivalent to the suspension functor $\Sigma$.
The Auslander--Reiten quiver of $\cc_n$ is depicted as
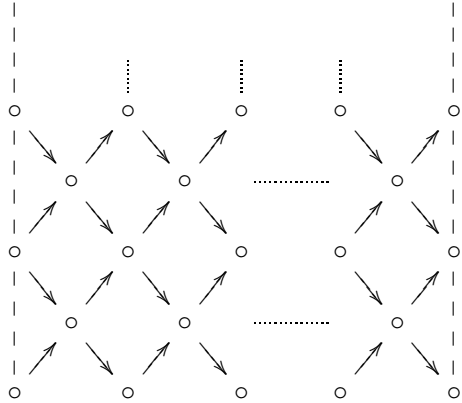
\[\xymatrix@C=0.5pc@R=1pc{&&&&&&&&\\
&&&&&&&&\\
\circ\ar[dr]\ar@{--}[dd]\ar@{--}[uu] && \circ\ar@{.}[u]\ar[dr] && \circ\ar@{.}[u] &&\circ\ar@{.}[u]\ar[dr] && \circ\ar@{--}[uu]\ar@{--}[dd]\\
& \circ\ar[dr]\ar[ur] &&\circ\ar[dr]\ar[ur]
&\ar@{.}[rr]&&& \circ\ar[dr]\ar[ur] &\\
\circ\ar[dr]\ar[ur]\ar@{--}[dd]&&\circ\ar[dr]\ar[ur]&&\circ&&\circ\ar[dr]\ar[ur]&&\circ\ar@{--}[dd]\\
& \circ\ar[dr]\ar[ur] &&\circ\ar[dr]\ar[ur]
&\ar@{.}[rr]&&& \circ\ar[ur]\ar[dr] &\\
\circ\ar[ur]&&\circ\ar[ur]&&\circ&&\circ\ar[ur]&&\circ
 }\]
where the leftmost and rightmost columns are identified.

The following lemma is clear.

\begin{lemma}\label{l:fundamental-domain}
On isomorphism classes of objects, the composite functor
$\ct_n\rightarrow\cd^b(\ct_n)\stackrel{\pi}{\rightarrow}\cc_n$ is
bijective. We will identify objects of $\ct_n$ and objects of
$\cc_n$ via this bijection. For $M$, $N$ in $\ct_n$,
\begin{eqnarray*}\Hom_{\cc_n}(M,N)&=&\Hom_{\ct_n}(M,N)\oplus\Ext^1_{\ct_n}(M,\tau^{-1}N),\\
\Hom_{\cc_n}(M,\Sigma N)&=&\Ext^1_{\ct_n}(M,N)\oplus
D\Ext^1_{\ct_n}(N,M).\end{eqnarray*}
\end{lemma}


\subsection{Maximal rigid objects of the cluster tube}
Let $\cc$ be a Krull--Schmidt 2-Calabi--Yau triangulated category
with suspension functor $\Sigma$. An object $M$ of $\cc$ is
\emph{rigid} if $\Hom_{\cc}(M,\Sigma M)=0$. It is \emph{maximal
rigid} if it is rigid and $\Hom_{\cc}(M\oplus N,\Sigma(M\oplus
N))=0$ implies that $N\in\add_{\cc}(M)$, the additive hull of $M$ in
$\cc$. It is \emph{cluster-tilting} if it is rigid and
$\Hom_{\cc}(M,\Sigma N)=0$ implies that $N\in\add_{\cc}(M)$. In view
of the second formula in Lemma~\ref{l:fundamental-domain}, an
indecomposable object of the cluster tube $\cc_n$ is rigid if and
only if it is rigid in $\ct_n$ if and only if it has length smaller
than or equal to $n-1$. In particular, the zero object is maximal
rigid in $\ct_1$. From now on, we assume $n\geq 2$.

Removing the vertex $n$ in the cyclic quiver
$\overrightarrow{\Delta}_n$, we obtain the quiver
$\overrightarrow{A}_{n-1}$ of type $A_{n-1}$ with linear
orientation. This yields a linear functor from
$\rep\overrightarrow{A}_{n-1}$ to $\ct_n$ which preserves the
Hom-spaces and the Ext$^1$-spaces. Composing this functor with the
functor from $\ct_n$ to $\cc_n$ described in
Lemma~\ref{l:fundamental-domain}, we obtain a functor $F$ from
$\rep\overrightarrow{A}_{n-1}$ to $\cc_n$. For $a=1,\ldots,n$, let
$F_{a}=\Sigma^{-a+1}\circ F$.

\begin{proposition}[Buan--Marsh--Vatne~\cite{BuanMarshVatne10} Proposition 2.6]\label{p:maximal-rigid-objects} An
object of $\cc_n$ is maximal rigid if and only if it is the image
under some $F_a$ of a tilting module in
$\rep\overrightarrow{A}_{n-1}$.
\end{proposition}

Note that each tilting module in $\rep\overrightarrow{A}_{n-1}$
contains as a direct summand the unique projective-injective
indecomposable module. Therefore, the Loewy length of a maximal
rigid object is $n-1$.

\begin{theorem}[Vatne~\cite{Vatne11} Theorem
2.1]\label{t:endo-algebra} Fix $a=1,\ldots,n$. Let $T$ be a basic
tilting module in $\rep\overrightarrow{A}_{n-1}$, $B\cong kQ/I$ be
the corresponding cluster-tilted algebra, where $I$ is an admissible
ideal of $kQ$. Then the endomorphism algebra $\End_{\cc_n}(F_a T)$
of $F_a T$ in $\cc_n$ is isomorphic to $k\tilde{Q}/\tilde{I}$, where
$\tilde{Q}$ is the quiver obtained from $Q$ by adding a loop
$\varphi$ at the vertex corresponding to the projective-injective
indecomposable module in $\rep\overrightarrow{A}_{n-1}$ and
$\tilde{I}$ is the ideal of $k\tilde{Q}$ generated by $I$ and
$\varphi^2$.
\end{theorem}
\begin{proof} We sketch a proof (for $a=1$), see~\cite{Vatne11} for the details. By the first formula in
Lemma~\ref{l:fundamental-domain}, we have
\[\End_{\cc}(FT)=\Hom_{\ct_n}(FT,FT)\oplus\Ext^1_{\ct_n}(FT,\tau^{-1}FT).\]
Suppose $T=T_0\oplus T_1$, where $T_1$ is injective and $T_0$ has no
injective direct summand. Then $\tau^{-1}FT_0\cong
F\tau_{A_{n-1}}^{-1}T_0$, where $\tau_{A_{n-1}}$ is the
Auslander--Reiten translation of $\rep\overrightarrow{A}_{n-1}$, and
hence
\begin{eqnarray*}
\End_{\cc}(FT)&=&\Hom_{\ct_n}(FT,FT)\oplus\Ext^1_{\ct_n}(FT,F\tau^{-1}_{A_{n-1}}T_0)\oplus\Ext^1_{\ct_n}(FT,\tau^{-1}FT_1)\\
&\cong&\Hom_{\rep\overrightarrow{A}_{n-1}}(T,T)\oplus\Ext^1_{\rep\overrightarrow{A}_{n-1}}(T,\tau^{-1}_{A_{n-1}}T_0)\oplus\Ext^1_{\ct_n}(FT,\tau^{-1}FT_1)\\
&\cong&\Hom_{\rep\overrightarrow{A}_{n-1}}(T,T)\oplus\Ext^1_{\rep\overrightarrow{A}_{n-1}}(T,\tau^{-1}_{A_{n-1}}T)\oplus\Ext^1_{\ct_n}(FT,\tau^{-1}FT_1)\\
&=&B\oplus\Ext^1_{\ct_n}(FT,\tau^{-1}FT_1),
\end{eqnarray*}
where $B$ is the cluster-tilted algebra corresponding to $T$, as in
the statement of the theorem. Let $\varphi$ be a nonzero element of
the space $\Ext^1_{\ct_n}((1,n),\tau^{-1}(1,n))$, which is a
1-dimensional subspace of $\Ext^1_{\ct_n}(FT,\tau^{-1}FT_1)$. The
square of $\varphi$ is clearly zero. Finally one checks that the
space $\Ext^1_{\ct_n}(FT,\tau^{-1}FT_1)$ has a basis each of which
factors through $\varphi$, and that there are no more relations.
\end{proof}

The object $(1,1)\oplus\ldots\oplus (1,n-1)$ is a typical maximal
rigid object. Its endomorphism algebra is the quotient of the path
algebra of the quiver
\[\xymatrix{\cdot \ar[r] & \cdot \ar[r] & \cdots \ar[r] & \cdot \ar@(ur,dr)[]^\varphi}\]
modulo the ideal generated by $\varphi^2$.

We define the \emph{wing} determined by a rigid object $(a,b)$ of
$\cc_n$ to be the additive hull of the indecomposable objects in the
triangle with vertices $(a,b)$, $(a,1)$ and $(a+b-1,1)$ of the
Auslander--Reiten quiver. For $a=1,\ldots,n$, the essential image of
the functor $F_a:\rep\overrightarrow{A}_{n-1}\rightarrow\cc_n$ is
exactly the wing of $(a,n-1)$. It follows from
Proposition~\ref{p:maximal-rigid-objects} that each maximal rigid
object of $\cc_n$ is in the wing of some $(a,n-1)$, and in this case
it has $(a,n-1)$ as a direct summand. The following lemma, which
appears in the proof of~\cite[Corollary 2.7]{BuanMarshVatne10} (see
also~\cite[Theorem 4.9]{Vatne11}), shows that a maximal rigid object
of $\cc_n$ is not a cluster-tilting object but not too far from
being one.

\begin{lemma}[Buan--Marsh--Vatne~\cite{BuanMarshVatne10}]\label{l:almost-tiltingness} Let $T$ be a maximal rigid object in the wing
of $(a,n-1)$ for $a=1,\ldots,n$. Then for an indecomposable object
$M$ of $\cc_n$, the Hom-space $\Hom_{\cc_n}(T,\Sigma M)$ vanishes if
and only if either $M$ is isomorphic to a direct summand of $T$ or
$M$ is isomorphic to $(a,sn-1)$ for some $s\geq 2$.
\end{lemma}

\section{Mutations of maximal rigid
objects}\label{s:mutations-of-maximal-rigid-objects}

Let $\cc$ be a 2-Calabi--Yau Krull--Schmidt triangulated category.
Let $T$ be a basic maximal rigid object. Let $R$ be an
indecomposable direct summand of $T$, and write $T=R\oplus \bar{T}$.
By~\cite[Theorem I.1.10 (a)]{BuanIyamaReitenScott09}, there is a
unique indecomposable object $R'$ of $\cc$ such that $R'$ is not
isomorphic to $R$ and $T'=R'\oplus \bar{T}$ is maximal rigid. There
are, up to isomorphism, unique triangles (called \emph{exchange
triangles})
\[\xymatrix{R'\ar[r]^f&B\ar[r]^g&R\ar[r]&\Sigma R'}\text{   and   }\xymatrix{R\ar[r]^{f'}&B'\ar[r]^{g'}&R'\ar[r]&\Sigma R}\]
with $f,f'$ being left $\add(\bar{T})$-approximations and $g,g'$
being right $\add(\bar{T})$-approximations. The procedure of
obtaining $T'$ from $T$ is called a \emph{mutation}. It is a
\emph{simple mutation} if $\dim_k\Hom_{\cc}(R,\Sigma
R')=\dim_k\Hom_{\cc}(R',\Sigma R)=1$, \confer~\cite{Keller08c}.

\subsection{Mutations of maximal rigid objects in the cluster tube}
Let $n\geq 2$ be an integer and let $\cc_n$ be the cluster tube of
rank $n$. In this subsection we will study mutations of maximal
rigid objects of $\cc_n$.

Recall that $\rep\overrightarrow{A}_{n-1}$ has a unique
indecomposable projective-injective object (up to isomorphism). Let
$\bar{M}$ be an
 almost complete basic tilting module in
$\rep\overrightarrow{A}_{n-1}$. According to~\cite[Proposition
2.3]{HappelUnger89}, $\bar{M}$ has precisely two complements if it
has the unique projective-injective module as a direct summand and
one complement otherwise. In the former case, let $N$ and $N'$ be
the two complements, then we say that $\bar{M}\oplus N$ and
$\bar{M}\oplus N'$ are related by a mutation. Let $a=1,\ldots,n$. By
Proposition~\ref{p:maximal-rigid-objects}, $F_a(\bar{M})$ is an
almost complete maximal rigid object in $\cc_n$, and hence it has
two complements in $\cc_n$ (see the beginning of this section).
Therefore we have

\begin{lemma}\label{l:mutation-tilting-module} Fix $a=1,\ldots,n$. Let $M$ and $M'$ be two basic tilting modules in $\rep\overrightarrow{A}_{n-1}$. Then $M$ and $M'$ are related by a mutation of
tilting modules if and only if $F_a M$ and $F_a M'$ are related by a
mutation of maximal rigid objects.
\end{lemma}

\begin{lemma}\label{l:simple-mutation}
Let $T$ be a basic maximal rigid object in the wing of $(a,n-1)$ for
$a=1,\ldots,n$. Suppose that $T'$ is a maximal rigid object in
$\cc_n$ related to $T$ by a mutation. Write $T=R\oplus\bar{T}$ and
$T'=R'\oplus\bar{T}$, where $R$, $R'$ are non-isomorphic
indecomposable objects. The following conditions are equivalent
\begin{itemize}
\item[i)] the mutation is simple,
\item[ii)] the length of $R$ is strictly smaller than $n-1$,
\item[iii)] the length of $R'$ is strictly smaller than $n-1$,
\item[iv)] $T'$ is in the wing of $(a,n-1)$.
\end{itemize}
\end{lemma}
\begin{proof}

Recall that $(a,n-1)$ is a direct summand of $T$, and all other
indecomposable direct summands of $T$ have length strictly smaller
than $n-1$. Thus, if the length of $R$ is $n-1$, then $R=(a,n-1)$,
and the length of $R'$ has to be $n-1$ as well. Since $R'$ is not
isomorphic to $R$, it follows that $R'=(a',n-1)$ for some
$a'=1,\ldots,n$ with $a'\neq a$. In this case,
$\Hom_{\cc_n}(R,\Sigma R')$ is 2-dimensional
(Lemma~\ref{l:fundamental-domain}), and $T'$ is not in the wing of
$(a,n-1)$. If the length of $R$ is strictly smaller than $n-1$, then
$(a,n-1)$ is a direct summand of $\bar{T}$, and hence a direct
summand of $T'$. In particular, $T'$ is in the wing of $(a,n-1)$ and
the length of $R'$ is strictly smaller than $n-1$. Let $M$ and $M'$
be tilting modules in $\rep\overrightarrow{A}_{n-1}$ such that
$T=F_a M$ and $T'=F_a M'$. By Lemma~\ref{l:mutation-tilting-module},
the tilting modules $M$ and $M'$ are related by a mutation.
By~\cite[Theorem 1.1]{HappelUnger89}, one of
$\Ext^1_{\ct_n}(R,R')=\Ext^1_{\overrightarrow{A}_{n-1}}(R,R')$ and
$\Ext^1_{\ct_n}(R',R)=\Ext^1_{\overrightarrow{A}_{n-1}}(R',R)$ is
1-dimensional and the other vanishes. Hence $\Hom_{\cc_n}(R,\Sigma
R')$ is 1-dimensional, that is, the mutation is simple.
\end{proof}

\begin{lemma}\label{l:maximal-rigid-general-and-standard} Fix
$a=1,\ldots,n$. A maximal rigid object of $\cc_n$ in the wing of
$(a,n-1)$ can be obtained from $(a,1)\oplus\ldots\oplus (a,n-1)$ by
a sequence of simple mutations.
\end{lemma}
\begin{proof} Suppose $T=F_a M$, where $M$ is a tilting module in
$\rep\overrightarrow{A}_{n-1}$. By~\cite[Corollary
2.2]{HappelUnger05}, $M$ can be obtained by a sequence of mutations
from $P$, the basic projective generator of
$\rep\overrightarrow{A}_{n-1}$. It follows from
Lemma~\ref{l:mutation-tilting-module} that $T=F_a M$ can be obtained
from $F_a P=(1,1)\oplus\ldots\oplus (1,n-1)$ by a sequence of
mutations such that each intermediate maximal rigid object is in the
wing of $(a,n-1)$. By Lemma~\ref{l:simple-mutation}, each
intermediate mutation is simple.
\end{proof}

In the above, we reduced the study of simple mutations to the study
of tilting modules in $\rep\overrightarrow{A}_{n-1}$. When the
mutation is not simple, we can explicitly describe the resulting
maximal rigid object. Let $T=(a,n-1)\oplus\bar{T}$ be a basic
maximal rigid object and $T'$ be the maximal rigid object obtained
from $T$ by mutating at $(a,n-1)$. Let $b$ be the maximal integer
such that $(a,b)$ is a direct summand of $\bar{T}$. Then $n-b-2$ is
the maximal integer $b'$ such that $(a+n-1-b',b')$ is a direct
summand of $\bar{T}$. Moreover, an indecomposable direct summand of
$\bar{T}$ is either in the wing of $(a,b)$ or in the wing of
$(a+b+1,n-b-2)$. These follow from
Proposition~\ref{p:maximal-rigid-objects} and the corresponding
results on tilting modules in $\rep\overrightarrow{A}_{n-1}$
(see~\cite[Section 4.1]{HappelRingel}). In~\cite{Vatne11}, the
triple $((a,n-1);(a,b),(a+b+1,n-b-2))$ is called a \emph{subwing
triple}. It is readily seen that $a+b+1$ (taken modulo $n$) is the
unique integer in $1,\ldots,n$ different from $a$ such that the wing
determined by $(a+b+1,n-1)$ contains $(a,b)$ and $(a+b+1,n-b-2)$. It
follows that $T'$ is isomorphic to $(a+b+1,n-1)\oplus \bar{T}$. The
subwing triple associated to $T'$ is
$((a+b+1,n-1);(a+b+1,n-b-2),(a,b))$. See the following picture:
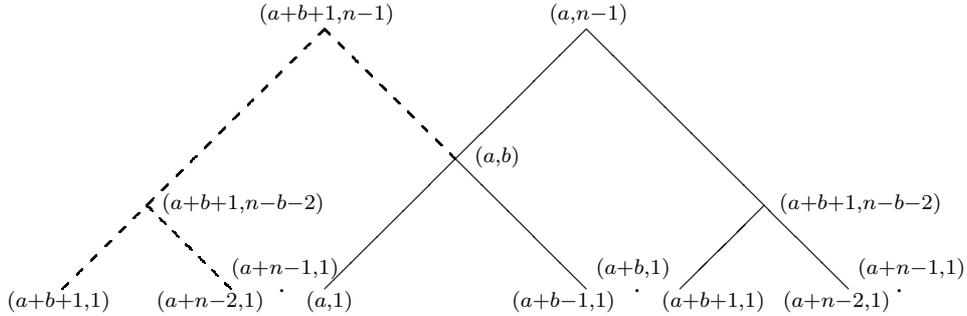
\[\hspace{30pt}{\setlength{\unitlength}{0.7pt}
\begin{picture}(500,180)
\dashline{5}(0,0)(140,140) \dashline{5}(140,140)(210,70)
\drawline(140,0)(280,140) \drawline(280,140)(420,0)
\dashline{5}(45,45)(90,0) \drawline(375,45)(330,0)
\put(115,-5){$\cdot$} \put(305,-5){$\cdot$} \put(445,-5){$\cdot$}
\drawline(210,70)(280,0)

\put(105,145){${\scriptstyle(a+b+1,n-1)}$}
\put(260,145){${\scriptstyle(a,n-1)}$}

\put(220,68){${\scriptstyle(a,b)}$}

\put(53,43){${\scriptstyle(a+b+1,n-b-2)}$}
\put(383,43){${\scriptstyle(a+b+1,n-b-2)}$}

\put(-30,-10){${\scriptstyle(a+b+1,1)}$}
\put(50,-10){${\scriptstyle(a+n-2,1)}$}
\put(90,8){${\scriptstyle(a+n-1,1)}$}
\put(130,-10){${\scriptstyle(a,1)}$}

\put(240,-10){${\scriptstyle(a+b-1,1)}$}
\put(285,8){${\scriptstyle(a+b,1)}$}
\put(320,-10){${\scriptstyle(a+b+1,1)}$}
\put(385,-10){${\scriptstyle(a+n-2,1)}$}
\put(425,8){${\scriptstyle(a+n-1,1)}$}
\end{picture}}\]
Here the vertices with the same label are identified.

\section{Objects finitely presented by a maximal rigid
object}\label{s:finitely-presented-objects}

 Fix an integer $n\geq 2$ and let $\cc_n$ be the
cluster tube of rank $n$.

In~\cite{Vatne11}, Vatne studied endomorphism algebras of maximal
rigid objects of $\cc_n$. Among other results, he proved that these
algebras are Gorenstein of Gorenstein dimension 1 except when $n=2$
in which case the algebras are symmetric, see~\cite[Proposition
3.3]{Vatne11}. In fact, he showed that these algebras are gentle,
and hence are Gorenstein by~\cite{GeissReiten05}. He described all
indecomposable modules in terms of strings, and as a consequence he
showed that these algebras are of finite representation type,
see~\cite[Theorem 3.8]{Vatne11}. In this section, we provide a
categorical explanation of the Gorenstein property and the
representation-finiteness. Very recently, it has been proved by
Zhou--Zhu in~\cite{ZhouZhu10} that the endomorphism algebra of a
maximal rigid object in \emph{any} 2-Calabi--Yau triangulated
category is Gorenstein of Gorenstein dimension at most 1.

\subsection{Objects finitely presented by a rigid object}\label{s:objects-finitely-presented-general}
Let $\cc$ be a Hom-finite Krull--Schmidt triangulated category over
the field $k$ with suspension functor $\Sigma$, and $T$ a rigid
object. An object $M$ of $\cc$ is \emph{finitely presented by $T$}
if there is a triangle in $\cc$
\[\xymatrix{T_1\ar[r]& T_0\ar[r]^{f}& M\ar[r] &\Sigma T_1}\]
with $T_0$, $T_1$ in $\add(T)$. The morphism $f$ is necessarily a
right $\add(T)$-approximation of $M$, and conversely, the cone of
any $\add(T)$-approximation of an object $M$ finitely presented by
$T$ belongs to $\add(\Sigma T)$. 
 Let $\pr(T)$ denote the
subcategory of $\cc$ of objects finitely presented by $T$.
Obviously, $\Sigma T$ belongs to $\pr(T)$.

Let $A$ be the endomorphism algebra of $T$. Let $\mod A$ denote the
category of finite-dimensional right modules over $A$.
\begin{lemma}\label{l:finite-presented-objects-and-modules} The functor $\Hom_{\cc}(T,?):\cc\rightarrow \mod A$
induces an equivalence of additive categories
\[\pr(T)/(\Sigma T)\stackrel{\sim}{\longrightarrow} \mod A,\]
where the category on the left is the additive quotient of $\pr(T)$
by the ideal generated by $\Sigma T$.
\end{lemma}
\begin{proof} This is a special case of the first statement of~\cite[Proposition 6.2 (3)]{IyamaYoshino08}, \confer also~\cite[Theorem 2.2]{BuanMarshReiten07},~\cite[Proposition 2.1]{KellerReiten07} and~\cite[Lemma
3.2]{Plamondon10}.
\end{proof}

Assume further that $\cc$ is 2-Calabi--Yau. The following
proposition is a special case of a more general result of Plamondon.

\begin{proposition}[Plamondon~\cite{Plamondon10} Proposition 2.7]\label{p:mutation-preserves-pr}
Let $T$ and $T'$ be two maximal rigid objects of $\cc$ related by a
simple mutation. Then $\pr(T)=\pr(T')$.
\end{proposition}

The following corollary of Proposition~\ref{p:mutation-preserves-pr}
generalizes~\cite[Theorem 4.2]{BuanMarshReiten07}, \cite[Proposition
2.2]{KellerReiten07}.

\begin{corollary}\label{c:nearly-morita-equivalence} Let $T$ and $T'$ be two
maximal rigid objects of $\cc$ related by a simple mutation. Then
the endomorphism algebras $\End_{\cc}(T)$ and $\End_{\cc}(T')$ are
related by a nearly Morita equivalence in the sense of
Ringel~\cite{Ringel07}.
\end{corollary}
\begin{proof} The proof is similar to that of~\cite[Theorem
4.2]{BuanMarshReiten07}. For completeness we provide it here.

We have the following diagram
\[\xymatrix{\pr(T)\ar@{=}[r]\ar[d]_{G_T=\Hom_{\cc}(T,?)} & \pr(T')\ar[d]^{G_{T'}=\Hom_{\cc}(T',?)}\\
\mod\End_{\cc}(T)&\mod\End_{\cc}(T').}\]

We assume that $T$ and $T'$ are basic and the mutation is at the
indecomposable direct summand $R$ of $T$. Suppose $T=R\oplus\bar{T}$
and $T'=R'\oplus\bar{T}$, where $R'$ is indecomposable. Clearly
$\Sigma T\oplus\Sigma R'=\Sigma T'\oplus\Sigma R$. The mutation
being simple implies that $G_T(\Sigma R')\cong S$, the simple
$\End_{\cc}(T)$-module corresponding to $R$, and similarly,
$G_{T'}(\Sigma R)\cong S'$, the simple $\End_{\cc}(T')$-module
corresponding to $R'$. Thus by
Lemma~\ref{l:finite-presented-objects-and-modules} we obtain the
following commutative diagram
\[\xymatrix{\pr(T)/(\Sigma T\oplus\Sigma R')\ar@{=}[r]\ar[d]_{\cong} & \pr(T')/(\Sigma T'\oplus\Sigma R)\ar[d]^{\cong}\\
\mod\End_{\cc}(T)/(S)\ar@{-->}[r]&\mod\End_{\cc}(T')/(S'),}\] where
the dashed arrow represents the desired nearly Morita equivalence.
\end{proof}

It is proved in~\cite{ZhouZhu10} that the two algebras
$\End_{\cc}(T)$ and $\End_{\cc}(T')$ as in
Corollary~\ref{c:nearly-morita-equivalence} have the same
representation type. Thus we have the following consequence of
Lemma~\ref{l:maximal-rigid-general-and-standard}.

\begin{corollary}
The endomorphism algebra of a maximal rigid object in the cluster
tube of rank $n$ is related to the algebra
\[k
(\xymatrix{\cdot \ar[r] & \cdot \ar[r] & \cdots \ar[r] & \cdot
\ar@(ur,dr)[]^\varphi})/(\varphi^2)~~\] by a sequence of nearly
Morita equivalences. In particular, the two algebra have the same
representation type.
\end{corollary}

\subsection{Rigid objects in the cluster
tube and the Gorenstein property}\label{s:rigid-is-presented}

 Let $\cc_n$ be the cluster tube of rank
$n$, and let $T$ be a maximal rigid object. In this subsection we
shall prove the following result.

\begin{proposition}\label{p:rigid-is-presented} Any indecomposable rigid object of $\cc_n$ is finitely presented by
$T$.
\end{proposition}

This result still holds even if we replace $\cc_n$ by any
2-Calabi--Yau triangulated category. This was first stated
in~\cite{BuanIyamaReitenScott09}, and a detailed proof can be found
in~\cite{ZhouZhu10}. Here we give a proof which relies on the
features of the cluster tube $\cc_n$.

We need some preparation. Let $M$, $N$ be two objects of $\ct_n$,
viewed as objects of $\cc_n$, and let $f:M\rightarrow N$ be a
morphism in $\cc_n$. According to the first formula in
Lemma~\ref{l:fundamental-domain}, $f$ has two components: $f_1\in
\Hom_{\ct_n}(M,N)$ and $f_2\in \Ext^1_{\ct_n}(M,\tau^{-1}N)$. Form
the triangle
\[\xymatrix{C\ar[r] & M\ar[r]^f & N \ar[r] &\Sigma C}.\] As shown in
Keller's proof of the main theorem (or rather of Theorem 9.9)
in~\cite{Keller05}, we have a triangle in $\cd(\ct_n)$
\[\xymatrix{\bigoplus(\tau^{-1}\circ\Sigma)^i C\ar[r] &\bigoplus(\tau^{-1}\circ\Sigma)^i M\ar[rr]^{\bigoplus(\tau^{-1}\circ\Sigma)^i f} && \bigoplus(\tau^{-1}\circ\Sigma)^i N \ar[r] & }\]
where all direct sums are over $\mathbb{Z}$. Taking cohomologies
gives a long exact sequence
\[\xymatrix{
\tau^{-1}M\ar[r]^{\tau^{-1}f_1} & \tau^{-1}N \ar[r]& C \ar[r] &
M\ar[r]^{f_1} & N ,}\] where the short exact sequence
\[\xymatrix{0\ar[r] & \cok(\tau^{-1}f_1)\ar[r] & C \ar[r] & \ker(f_1)\ar[r] &
0}\] is induced from $f_2$ by the inclusion
$\ker(f_1)\hookrightarrow M$ and the quotient
$\tau^{-1}N\twoheadrightarrow \cok(\tau^{-1}f_1)$. Namely, we have
the following commutative diagram
\[\xymatrix{
f_2: & 0\ar[r] & \tau^{-1}N\ar[r]\ar@{->>}[d] & E\ar[r]\ar[d] & M\ar[r]\ar@{=}[d] & 0\\
& 0\ar[r] & \cok(\tau^{-1}f_1)\ar[r] & E'\ar[r] & M\ar[r] & 0\\
& 0\ar[r] & \cok(\tau^{-1}f_1)\ar[r]\ar@{=}[u] & C\ar[u] \ar[r] &
\ker(f_1)\ar[r]\ar@{^{(}->}[u] & 0~~,}\] where the square in the
upper-left corner is a pushout and the square in the lower-right
corner is a pullback.
 As a consequence, the Loewy length
of $C$ is smaller than or equal to the sum of the Loewy lengths of
$M$ and $N$. Namely, we have proved the following lemma.

\begin{lemma}\label{l:loewy-length} Let $\xymatrix{C\ar[r] & M\ar[r] & N\ar[r] &\Sigma C}$
be a triangle in $\cc_n$. Then
\[\LL(C)\leq \LL(M)+\LL(N).\]
In particular, for an object $M$ of $\pr(T)$ we have $\LL(M)\leq
2(n-1)$.
\end{lemma}

\begin{proof}[Proof of Proposition~\ref{p:rigid-is-presented}]
Let $M$ be an indecomposable rigid object. Then $\LL(M)\leq n-1$.
Let $T_0\stackrel{f}{\rightarrow} M$ be a right
$\add_{\cc_n}(T)$-approximation of $M$, and form the triangle
\[\xymatrix{T_1\ar[r] & T_0\ar[r]^f & M\ar[r] &\Sigma T_1}.\]
Then $\Hom_{\cc_n}(T,\Sigma T_1)=0$ and $\LL(T_1)\leq 2(n-1)$. Hence
it follows from Lemma~\ref{l:almost-tiltingness} that $T_1$ belongs
to $\add_{\cc_n}(T)$ since $\LL(a,sn-1)=sn-1 > 2(n-1)$ for all
$a=1,\ldots,n$ and all $s\geq 2$. This finishes the proof of
Proposition~\ref{p:rigid-is-presented}.
\end{proof}

 As an application of Proposition~\ref{p:rigid-is-presented}, we
have the following result which was first proved by Vatne
in~\cite{Vatne11}.

\begin{corollary}[Vatne~\cite{Vatne11} Proposition 3.3] The endomorphism algebra $A=\End_{\cc_n}(T)$ of
$T$ is Gorenstein of Gorenstein dimension $1$ unless $n=2$, in which
case $A$ is symmetric.
\end{corollary}
\begin{proof} When $n=2$, there are two basic maximal rigid objects up to isomorphism, namely, $(1,1)$ and $(2,1)$. Their
 endomorphism algebras are isomorphic to $k[x]/(x^2)$, which is symmetric.

We assume $n\geq 3$. The functor $\Hom_{\cc_n}(T,?)$ takes $T$ to
$A$ and takes $\Sigma^2 T$ to $D(A)$ due to the 2-Calabi--Yau
property. It follows from Proposition~\ref{p:rigid-is-presented}
that $\Sigma^2 T$ is finitely presented by $T$ and $T$ is finitely
presented by $\Sigma T$. Therefore $A$ has injective dimension at
most 1 and $D(A)$ has projective dimension at most 1. It remains to
show that $D(A)$ is not projective, or equivalently, $A$ is not
injective. It follows from
Lemma~\ref{l:finite-presented-objects-and-modules} that for an
object $M$ of $\pr(T)$ the $A$-module $\Hom_{\cc_n}(T,M)$ is
projective if and only if $M$ belongs to $\add_{\cc_n}(\Sigma
T\oplus T)$. Notice that $\add_{\cc_n}(\Sigma T\oplus T)$ contains
precisely two indecomposable objects of Loewy length $n-1$:
$(a,n-1)$ and $(a-1,n-1)$. Therefore the direct summand $(a-2,n-1)$
of $\Sigma^2 T$, and hence $\Sigma^2 T$, does not belong to
$\add_{\cc_n}(\Sigma T\oplus T)$, showing that $D(A)$ is not
projective.
\end{proof}

\subsection{Objects finitely presented by a maximal rigid object and representation-finiteness}\label{s:objects-finitely-presented-cluster-tube}
Fix a basic maximal rigid object $T$ of the cluster tube $\cc_n$. In
the preceding subsection we showed that each indecomposable rigid
object lies in $\pr(T)$. In this subsection, we will determine all
the indecomposable objects in $\pr(T)$. In particular, we will show
that there are only finitely many of them. In view of
Lemma~\ref{l:finite-presented-objects-and-modules}, we deduce that
the endomorphism algebra of $T$ is of finite representation type.

Following Vatne~\cite[Section 4]{Vatne11}, we define $\cf$ to be the
set of indecomposable objects $(a,b)$ satisfying either (1) $b\leq
n-1$, \ie~ $(a,b)$ is rigid, or (2) $n\leq b\leq 2(n-1)$ and
$a+b\leq 2n-1$. Here is the picture of $\cf$ for $n=4$ (the black
vertices belong to $\cf$ while the white ones do not), see
also~\cite[Figure 6]{Vatne11}:
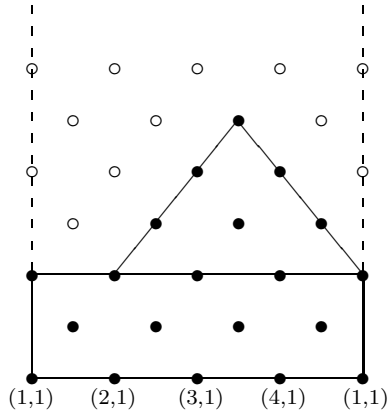
\[{\setlength{\unitlength}{0.7pt}
\begin{picture}(200,210)
\put(0,80){$
\begin{array}{ccccccccc}
\circ&&\circ&&\circ&&\circ&&\circ\\
&\circ&&\circ&&\bullet&&\circ&\\
\circ&&\circ&&\bullet&&\bullet&&\circ\\
&\circ&&\bullet&&\bullet&&\bullet&\\
\bullet&&\bullet&&\bullet&&\bullet&&\bullet\\
&\bullet&&\bullet&&\bullet&&\bullet&\\
\bullet&&\bullet&&\bullet&&\bullet&&\bullet\end{array}$}
\drawline(11,-1)(11,55) \drawline(11,-1)(188,-1)
 \drawline(188,-1)(188,55)
\drawline(11,55)(188,55)

\drawline(55,55)(121,138) \drawline(121,138)(188,55)

\dashline{5}(11,-1)(11,200) \dashline{5}(188,-1)(188,200)

\put(-2,-15){${\scriptstyle(1,1)}$}
\put(177,-15){${\scriptstyle(1,1)}$}
\put(42,-15){${\scriptstyle(2,1)}$}
\put(89,-15){${\scriptstyle(3,1)}$}
\put(133,-15){${\scriptstyle(4,1)}$}
\end{picture}}\]

The main result of this subsection is
\begin{proposition}\label{p:finitely-presented-objects} Suppose that $T$ is in the wing of $(a,n-1)$ for
$a=1,\ldots,n$. Then an indecomposable object of $\cc_n$ belongs to
$\pr(T)$ if and only if it lies in $\Sigma^{-(a-1)}\cf$.
\end{proposition}

This shows that generally the category of finitely presented objects
by a maximal rigid object is not preserved under mutation. Let
$A=\End_{\cc_n}(T)$ be the endomorphism algebra of $T$. We have the
following results due to Vatne, who proved it using different
methods.

\begin{corollary}\label{c:representation-finiteness}
\begin{itemize}
\item[a)]{\rm (Vatne~\cite[Theorem 4.8 (2)]{Vatne11})} The functor
$\Hom_{\cc_n}(T,?)$ induces a bijection between
$\Sigma^{-(a-1)}\cf\backslash\add (\Sigma T)$ and isoclasses of
indecomposable $A$-modules.
\item[b)]{\rm (Vatne~\cite[Theorem 3.8]{Vatne11})}  The algebra $A$ is representation-finite. The number of
indecomposable $A$-modules is $\frac{3}{2}n^2-\frac{5}{2}n+1$.
\end{itemize}
\end{corollary}
\begin{proof}
a) follows from Lemma~\ref{l:finite-presented-objects-and-modules}
and Proposition~\ref{p:finitely-presented-objects}.

b) The number of objects of $\Sigma^{-(a-1)}\cf$ is
\[n(n-1)+\frac{n(n-1)}{2}=\frac{3n(n-1)}{2}.\] Therefore, the number of
indecomposable objects in $\mod A$ is
\[\frac{3n(n-1)}{2}-(n-1)=\frac{3}{2}n^2-\frac{5}{2}n+1.\]
\end{proof}

The rest of this subsection is devoted to proving
Proposition~\ref{p:finitely-presented-objects}. By
Proposition~\ref{p:mutation-preserves-pr} and
Lemma~\ref{l:maximal-rigid-general-and-standard}, we have
$\pr(T)=\pr((a,1)\oplus\ldots\oplus (a,n-1))$. Observing that the
validity of the statement is preserved under the suspension functor
$\Sigma$, we reduce the proof to the case $T=(1,1)\oplus\ldots\oplus
(1,n-1)$.

 Suppose that $T =(1,1)\oplus\ldots\oplus (1,n-1)$. In view
of Lemma~\ref{l:loewy-length} and
Proposition~\ref{p:rigid-is-presented}, we only need to consider the
objects $(a,b)$ with $n\leq b\leq 2(n-1)$. As pointed out in
Section~\ref{s:objects-finitely-presented-general}, $(a,b)$ belongs
to $\pr(T)$ if and only if the cone of a right
$\add_{\cc_n}(T)$-approximation of $(a,b)$ belongs to
$\add_{\cc_n}(\Sigma T)$. We will find a right
$\add_{\cc_n}(T)$-approximation of $(a,b)$ and determine whether its
cone is in $\add_{\cc_n}(\Sigma T)$. It is easy to see that any
morphism from $(1,b')$ with $b'\leq n-2$ to $(a,b)$ factors through
$(1,n-1)$, and hence a right $\add_{\cc_n}((1,n-1))$-approximation
of $(a,b)$ is a right $\add_{\cc_n}(T)$-approximation. We have
\begin{eqnarray*}
\Hom_{\cc_n}((1,n-1),(a,b))&=&\Hom_{\ct_n}((1,n-1),(a,b))\oplus\Ext^1_{\ct_n}((1,n-1),(a+1,b))~~,
\end{eqnarray*}
where
\begin{eqnarray*}
\Hom_{\ct_n}((1,n-1),(a,b))&=&\begin{cases} k & \text{ if } a\neq n\\ 0 & \text{ if } a=n \end{cases},\\
\Ext^1_{\ct_n}((1,n-1),(a+1,b))&=&\begin{cases} k & \text{ if }
a+b\neq 2n-1\\ 0 & \text{ if } a+b=2n-1\end{cases}.
\end{eqnarray*}
We divide the problem into four cases.

Case 1: $a=n$ and $a+b=2n-1$. This is not possible since $b\geq n$.

Case 2: $a=n$ and $a+b\neq 2n-1$, \ie ~$a=n$. In this case,
$\tau^{-1}(n,b)=(1,b)$. Let $f_2$ be a basis element of
$\Ext^1_{\ct_n}((1,n-1),(1,b))$:
\[f_2:\xymatrix{0\ar[r]&(1,b)\ar[r]&(1,2n-1)\oplus (1,b-n)\ar[r]& (1,n-1)\ar[r] & 0}.\] Then $f_2$ is a right
$\add_{\cc_n}(T)$-approximation of $(n,b)$. Form the triangle in
$\cc_n$
\[\xymatrix{C\ar[r] & (1,n-1)\ar[r]^(0.55){f_2} & (n,b)\ar[r] &\Sigma C}.\] As
explained in the proof of Lemma~\ref{l:loewy-length}, we have a long
exact sequence
\[\xymatrix{
(2,n-1)\ar[r]^{0} & (1,b) \ar[r]& C \ar[r] & (1,n-1)\ar[r]^{0} &
(n,b) ,}\] where the class of the short exact sequence in the middle
is exactly $f_2$. Therefore $C\cong (1,2n-1)\oplus (1,b-n)$ does not
belong to $\add_{\cc_n}(T)$, and hence $(n,b)$ does not belong to
$\pr(T)$.

Case 3: $a\neq n$ and $a+b=2n-1$. Let $f_1$ be a basis element of
$\Hom_{\ct_n}((1,n-1),(a,b))$. Then $f_1$ is a right
$\add_{\cc_n}(T)$-approximation of $(a,b)$. Form the triangle in
$\cc_n$ \[\xymatrix{C\ar[r] & (1,n-1)\ar[r]^(0.55){f_1} &
(a,b)\ar[r] &\Sigma C}.\] We obtain a long exact sequence
\[\xymatrix{
(2,n-1)\ar[r]^{\tau^{-1}f_1} & (a+1,b) \ar[r]& C \ar[r] &
(1,n-1)\ar[r]^{f_1} & (a,b) ,}\] where the short exact sequence
\[\xymatrix{0\ar[r] & \cok(\tau^{-1}f_1)\ar[r] & C \ar[r] & \ker(f_1)\ar[r] &
0}\] splits. It follows that $C\cong (1,a-1)\oplus (1,n-1)$ belongs
to $\add_{\cc_n}(T)$, and hence $(a,b)$ belongs to $\pr(T)$.

Case 4: $a\neq n$ and $a+b\neq 2n-1$. Let $f_1$ be a basis element
of $\Hom_{\ct_n}((1,n-1),(a,b))$ and $f_2$ be a basis element of
$\Ext^1_{\ct_n}((1,n-1),(a+1,b))$:
\[f_2:\xymatrix@C=0.7pc{0\ar[r] & (a+1,b)\ar[r] &
(a+1,n\lfloor\frac{a+b}{n}\rfloor+n-a-1)\oplus
(1,a+b-n\lfloor\frac{a+b}{n}\rfloor)\ar[r] & (1,n-1)\ar[r] & 0.}\]
Then $f=(f_1,f_2):(1,n-1)\oplus (1,n-1)\rightarrow (a,b)$ is a right
$\add_{\cc_n}(T)$-approximation of $(a,b)$. Form the triangle in
$\cc_n$ \[\xymatrix{C\ar[r] & (1,n-1)\oplus (1,n-1)\ar[r]^(.7)f &
(a,b)\ar[r] &\Sigma C}.\] We obtain a long exact sequence
\[\xymatrix{
(2,n-1)\oplus (2,n-1)\ar[r]^(0.65){(\tau^{-1}f_1,0)} & (a+1,b)
\ar[r]& C \ar[r] & (1,n-1)\oplus (1,n-1)\ar[r]^(0.67){(f_1,0)} &
(a,b) ,}\] where the short exact sequence
\[\xymatrix{0\ar[r] & \cok((\tau^{-1}f_1,0))\ar[r] & C \ar[r] & \ker((f_1,0))\ar[r] &
0}\] \ie
\[\xymatrix{0\ar[r] & (1,a+b-n)\ar[r] & C \ar[r] & (1,a-1)\oplus (1,n-1)\ar[r] &
0}\] is induced from the short exact sequence
\[(0,f_2):\xymatrix{0\ar[r] & (a+1,b)\ar[r] &
(1,n-1)\oplus (a+1,n\lfloor\frac{a+b}{n}\rfloor+n-a-1)\\
& \oplus (1,a+b-n\lfloor\frac{a+b}{n}\rfloor)\ar[r] &
(1,n-1)\oplus(1,n-1)\ar[r] & 0.}\] It follows that $C\cong
(1,a-1)\oplus(1,n\lfloor\frac{a+b}{n}\rfloor-1)\oplus
(1,a+b-n\lfloor\frac{a+b}{n}\rfloor)$. Therefore $(a,b)$ belongs to
$\pr(T)$ if and only if $C$ belongs to $\add_{\cc_n}(T)$ if and only
if $\lfloor\frac{a+b}{n}\rfloor<2$ if and only if $a+b<2n-1$ (recall
that in this case $a+b\neq 2n-1$). This completes the proof of
Proposition~\ref{p:finitely-presented-objects}.

\section{Derived equivalence
classification}\label{s:derived-equivalence}

Fix an integer $n\geq 2$, and let $\cc_n$ be the cluster tube of
rank $n$. In this section, we provide a derived equivalence
classification for endomorphism algebras of maximal rigid objects of
$\cc_n$. This classification is analogous to that of
Buan--Vatne~\cite{BuanVatne08} for cluster-tilted algebras of type
$A$.

Let $T$ be a basic maximal rigid object in $\cc_n$. We also view $T$
as a basic tilting module in $\rep\overrightarrow{A}_{n-1}$, \confer
Proposition~\ref{p:maximal-rigid-objects}. Let $B\cong kQ/I$ be the
cluster-tilted algebra corresponding to $T$, where $I$ is an
admissible ideal of $kQ$. Recall from Theorem~\ref{t:endo-algebra}
that the endomorphism algebra $A=\End_{\cc_n}(T)$ of $T$ in $\cc_n$
is isomorphic to $k\tilde{Q}/\tilde{I}$, where $\tilde{Q}$ is the
quiver obtained from $Q$ by adding a loop $\varphi$ at the vertex
$c$ corresponding to the projective-injective indecomposable module
in $\rep\overrightarrow{A}_{n-1}$ and $\tilde{I}$ is the ideal of
$k\tilde{Q}$ generated by $I$ and $\varphi^2$. We denote
$\gamma_c(Q)=\tilde{Q}$ and $\delta_c(\tilde{Q})=Q$.

Following~\cite{Vatne11}, we give a description of the quivers
$\tilde{Q}$. The quivers $\tilde{Q}$ are exactly the quivers with
$n-1$ vertices and satisfying the following
\begin{itemize}
\item all non-trivial minimal cycles of length at least 2 in the underlying graph are
oriented and of length 3 (in particular, there are no multiple
arrows or 2-cycles),
\item any vertex has at most four neighbours,
\item if a vertex has four neighbours, then two of its adjacent
arrows belong to one 3-cycle, and the other two belong to another
3-cycle,
\item if a vertex has three neighbours, then two of its adjacent
arrows belong to a 3-cycle, and the third does not belong to any
3-cycle,
\item there is precisely one loop $\varphi$, at a vertex $c$ which  has zero neighbour (this happens when and only when $n=2$), or has one neighbour, or has two neighbours and
is traversed by a 3-cycle.
\end{itemize}
Let $\widetilde{\cq}_{n-1}$ denote the set of such quivers. For a
quiver $\tilde{Q}\in\widetilde{\cq}_{n-1}$, let $I_{\tilde{Q}}$ be
the ideal of $k\tilde{Q}$ generated by the square of the unique loop
and all paths of length 2 involved in a 3-cycle. The following is a
corollary of Theorem~\ref{t:endo-algebra} and the description of the
relations of cluster-tilted algebras of type $A$
(\confer~\cite{CalderoChapotonSchiffler06},~\cite{BuanVatne08}).

\begin{corollary}\label{c:endo-algebra}
An algebra is the endomorphism algebra of a maximal rigid object of
$\cc_n$ if and only if it is isomorphic to
$k\tilde{Q}/I_{\tilde{Q}}$ for some
$\tilde{Q}\in\widetilde{\cq}_{n-1}$. In particular, the endomorphism
algebra of a maximal rigid object of $\cc_n$ is determined by its
quiver.
\end{corollary}

For a quiver in $\widetilde{\cq}_{n-1}$, we will always denote by
$c$ the vertex where the unique loop lies.

\begin{proposition}\label{p:change-of-quiver} Assume $n\geq 3$. Let $T$ and $T'$ be two basic maximal rigid objects of
$\cc_n$ related by a mutation. Let $\tilde{Q}$ and $\tilde{Q}'$ be
the quivers of $\End_{\cc_n}(T)$ and $\End_{\cc_n}(T')$
respectively. Assume that $T$ is in the wing of $(a,n-1)$.
\begin{itemize}
\item[a)] If the mutation is simple, then the quivers $\tilde{Q}$ and
$\tilde{Q}'$ are related by a Fomin--Zelevinsky mutation.
\item[b)] If the mutation is not simple, there are two cases
\begin{itemize}
\item[1)] if the vertex $c$ has one neighbour, then $\tilde{Q'}$
is obtained from $\tilde{Q}$ by reversing the unique arrow adjacent
to $c$;
\item[2)] if the vertex $c$ has two neighbours, then $\tilde{Q'}$
is obtained from $\tilde{Q}$ by reversing all arrows in the unique
3-cycle traversing $c$. 
\end{itemize}
\end{itemize}

\end{proposition}
\begin{proof} a) Assume that the mutation is simple, and the mutation is performed at $i$.
Then there are no loops or 2-cycles at $i$. The assertion follows
from a local version of~\cite[Theorem I.1.6]{BuanIyamaReitenScott09}
(note that the proof of ~\cite[Theorem
I.1.6]{BuanIyamaReitenScott09} is local).

b) Assume that the mutation is not simple. By
Lemma~\ref{l:simple-mutation}, the mutation is performed at the
direct summand $(a,n-1)$ of $T$, correspondingly, the vertex $c$ of
$\tilde{Q}$. In view of the exchange triangles, the arrows of
$\tilde{Q}$ adjacent to $c$ are reversed. As in the proof of
~\cite[Theorem I.1.6]{BuanIyamaReitenScott09}, it remains to
consider the situation where we have arrows $j\rightarrow
c\rightarrow j'$. If $c$ has one neighbour in $\tilde{Q}$, it also
has only one neighbour in $\tilde{Q}'$ and there are no subquivers
of the form $j\rightarrow c\rightarrow j'$. In this case, to obtain
$\tilde{Q}'$ from $\tilde{Q}$ we only need to reverse the unique
arrow adjacent to $c$. If $c$ has two neighbours, say, $j$ and $j'$,
in $\tilde{Q}$, then $j$ and $j'$ are also the only neighbours of
$c$ in $\tilde{Q}'$. Moreover, in $\tilde{Q}$, there is precisely
one subquiver of the form $j\rightarrow c\rightarrow j'$. Since $c$
is traversed by a 3-cycle in $\tilde{Q}$, there is an arrow
$j\leftarrow j'$. As explained above, in $\tilde{Q}'$ there is a
subquiver of the form $j\leftarrow c\leftarrow j'$. Since
$\tilde{Q}'\in\widetilde{\cq}_{n-1}$, it follows that in
$\tilde{Q}'$, the vertex $c$ is also traversed by a 3-cycle, and
hence there is an arrow $j\rightarrow j'$. Therefore, to obtain
$\tilde{Q}'$ from $\tilde{Q}$, all the three arrows in the 3-cycle
traversing $c$ are reversed.
\end{proof}

Proposition~\ref{p:change-of-quiver} shows that the quiver
$\tilde{Q}'$ only depends on the quiver $\tilde{Q}$ and the vertex
$i$ at which the mutation is taken and does not depend on the choice
of $T$. We will write $\tilde{Q}'=\mu_i(\tilde{Q})$, and by abuse of
language we call it the \emph{mutation} of $\tilde{Q}$ at $i$.

\begin{lemma}\label{l:mutation-and-loopization}
Let $\tilde{Q}$ be a quiver in $\widetilde{\cq}_{n-1}$ and $i$ a
vertex of $\tilde{Q}$. If $i$ is different from $c$ or $i=c$ has
only one neighbour, then $\mu_i\gamma_c(Q)=\gamma_c\mu_i(Q)$, where
$Q=\delta_c(\tilde{Q})$.
\end{lemma}
\begin{proof}
When $i$ is different from $c$, the statement follows from
Proposition~\ref{p:change-of-quiver} a), the definition of
Fomin--Zelevinsky mutation and the definition of $\gamma_c$. When
$i=c$ has only one neighbour, the statement follows from
Proposition~\ref{p:change-of-quiver} b) and the definition of
$\gamma_c$.
\end{proof}

Next we give a sufficient condition for the endomorphism algebras of
two neighbouring maximal rigid objects to be derived equivalent. It
is worth noting that in general the nearly Morita equivalence in
Corollary~\ref{c:nearly-morita-equivalence} is not a consequence of
the derived equivalence.

\begin{proposition}\label{p:derived-equivalence}
Let $T$ and $T'$ be two basic maximal rigid objects of $\cc_n$
related by a mutation. Then their endomorphism algebras are related
by a derived equivalence if the corresponding mutation of quivers
does not change the number of 3-cycles.
\end{proposition}

\begin{proof}
Suppose $T=R\oplus\bar{T}$ and $T'=R'\oplus\bar{T}$  with $R$ and
$R'$ indecomposable. We first assume that $T$ and $T'$ are related
by a simple mutation. By Proposition~\ref{p:change-of-quiver}, up to
symmetry the quivers of $\End_{\cc_n}(T)$ and $\End_{\cc_n}(T')$
locally look like
\[\xymatrix@!=0.5pc{&&\cdot\ar[rr]&&\cdot\ar[dl]\\
(a)&&&R\ar[ul]\ar[dr]&\\
&&\cdot\ar[ur]&&\cdot\ar[ll]}
\qquad\qquad\xymatrix@!=0.5pc{\cdot\ar[dr]&&\cdot\ar[dd]\\
&R'\ar[ur]\ar[dl]&\\
\cdot\ar[uu]&&\cdot\ar[ul]}\]
\[\xymatrix@!=0.5pc{&&\cdot\ar[rr]&&\cdot\ar[dl]\\
(b)&&&R\ar[ul]&\\
&&\cdot\ar[ur]&&}
\qquad\qquad\xymatrix@!=0.5pc{\cdot\ar[dr]&&\cdot\\
&R'\ar[ur]\ar[dl]&\\
\cdot\ar[uu]&&}\]
\[\xymatrix@!=0.5pc{(c)&&\cdot\ar[rr]&&\cdot\ar[dl]\\
&&&R\ar[ul]&}
\qquad\qquad\xymatrix@!=0.5pc{\cdot\ar[dr]&&\cdot\\
&R'\ar[ur]&}\]
\[\xymatrix@!=0.5pc{(d)&&\cdot\ar[r]&R&\cdot\ar[l]\\
&&&\End_{\cc_n}(T)&}
\qquad\qquad\xymatrix@!=0.5pc{\cdot&R'\ar[l]\ar[r]&\cdot\\
&\End_{\cc_n}(T')&}\]  It follows from
Corollary~\ref{c:endo-algebra} that a linear combination $\sum_j
\lambda_j p_j$ of paths is nonzero if and only if there is some $j$
such that $p_j$ is nonzero and $\lambda_j$ is nonzero. Moreover, in
cases (a) (b) and (d), for any nonzero path $p$ starting at $R$,
there is an arrow $\beta$ ending at $R$ such that $p\beta$ is
nonzero, and for any nonzero path $p'$ ending at $R'$, there is an
arrow $\alpha'$ starting at $R'$ such that $\alpha'p'$ is nonzero;
while in case (c), the two quivers have different numbers of
3-cycles. It follows from~\cite[Proposition 2.3, Theorem
4.2]{Ladkani10} that in cases (a), (b) and (d) the algebras
$\End_{\cc_n}(T)$ and $\End_{\cc_n}(T')$ are derived equivalent (via
a BB-tilting module).

Now we assume that $T$ and $T'$ are related by a non-simple
mutation. We claim that the endomorphism algebra of $T$ is derived
equivalent to that of $T'$. As in the proof of
Proposition~\ref{p:change-of-quiver} b), let
$((a,n-1);(a,b),(a+b+1,n-b-2))$ and
$((a+b+1,n-1);(a+b+1,n-b-2),(a,b))$ be the subwing triples
associated to $T$ and $T'$ respectively. There are three cases:

Case 1: $b=0$. In this case, locally at $(a,n-1)$ and $(a+1,n-1)$ we
have
\[\xymatrix@!=0.5pc{{\scriptscriptstyle (a,n-1)}\ar@(ur,ul)_{\varphi}\ar[d]_{\alpha}\\ {\scriptscriptstyle (a+1,n-2)}}
\qquad\qquad \xymatrix@!=0.5pc{{\scriptscriptstyle
(a+1,n-1)}\ar@(ul,ur)^{\varphi^{\star}}\\
{\scriptscriptstyle (a+1,n-2)}\ar[u]^{\alpha^{\star}}}
\] The square of the loops $\varphi^2$ and $\varphi^{\star 2}$ are the (local) relations.
The two exchange triangles are
\[\xymatrix{(a,n-1)\ar[r]^(0.32){f={\alpha\varphi\choose\alpha}}&(a+1,n-2)\oplus(a+1,n-2)\ar[rr]^(0.6){g=(\alpha^{\star},\varphi^{\star}\alpha^{\star})}&&(a+1,n-1)\ar[r]&},\]
\[\xymatrix{(a+1,n-1)\ar[r]&0\ar[r]&(a,n-1)\ar[r]&}.\]
A morphism in $\Hom_{\cc_n}(T,(a,n-1))$ is a linear combination of
paths ending at $(a,n-1)$, \ie~$\lambda_1
id_{(a,n-1)}+\lambda_2\varphi$. Such a combination is sent by
$\Hom_{\cc_n}(T,f)$ to
$\lambda_1{\alpha\varphi\choose\alpha}+\lambda_2{0\choose\alpha\varphi}$,
which is zero if and only if $\lambda_1=\lambda_2=0$. Therefore,
$\Hom_{\cc_n}(T,f)$ is injective. Dually, one shows that
$\Hom_{\cc_n}(g,T')$ is injective. It follows from~\cite[Lemma
3.4]{HuXi08} that $\End_{\cc_n}(T)$ and $\End_{\cc_n}(T')$ are
derived equivalent.

Case 2: $b=n-2$. This case is dual to Case 1.

Case 3: $b\neq 0$ and $b\neq n-2$. Locally we have
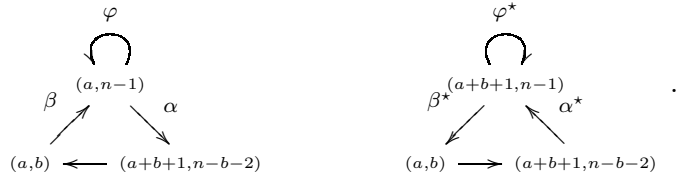
\[\xymatrix@!=0.5pc{&{\scriptscriptstyle (a,n-1)}\ar@(ur,ul)_{\varphi}\ar[dr]^{\alpha}&\\
{\scriptscriptstyle (a,b)}\ar[ur]^{\beta}&&{\scriptscriptstyle
(a+b+1,n-b-2)}\ar[ll]}\qquad\qquad
\xymatrix@!=0.5pc{&{\scriptscriptstyle (a+b+1,n-1)}\ar@(ul,ur)^{\varphi^{\star}}\ar[dl]_{\beta^{\star}}&\\
{\scriptscriptstyle (a,b)}\ar[rr]&&{\scriptscriptstyle
(a+b+1,n-b-2)}\ar[ul]_{\alpha^{\star}}}.
\] The (local) relations are the squares of the loops and the compositions
of two consecutive arrows in the same 3-cycles. The exchange
triangles are
\[\xymatrix{(a+b+1,n-1)\ar[r]^{\beta^{\star}\varphi^{\star}\choose\beta^{\star}}&(a,b)\oplus (a,b)\ar[r]^{(\beta,\varphi\beta)}&(a,n-1)\ar[r]&}\]
\[\xymatrix{(a,n-1)\ar[r]^(0.25){f={\alpha\varphi\choose\varphi}}&(a+b+1,n-b-2)\oplus(a+b+1,n-b-2)\ar[r]^(0.7){g=(\alpha^{\star},\varphi^{\star}\alpha^{\star})}&(a+b+1,n-1)\ar[r]&}\]
A nonzero path ending at $(a,n-1)$ is of the form $\beta p$ or
$\varphi\beta p$, where $p$ is a nonzero path ending at $(a,b)$ and
not passing through $(a+b+1,n-b-2)$. Thus a morphism of
$\Hom_{\cc_n}(T,(a,n-1))$ is of the form $\sum_p(\lambda_1^p \beta
p+\lambda_2^p\varphi\beta p)$. Such a morphism is sent by
$\Hom_{\cc_n}(T,f)$ to $\sum_p{\lambda_1^p\alpha\varphi\beta
p\choose\lambda_2^p\alpha\varphi\beta p}$, which is zero if and only
if all coefficients are zero. Therefore, $\Hom_{\cc_n}(T,f)$ is
injective. Dually, one shows that $\Hom_{\cc_n}(g,T')$ is injective.
Again it follows from~\cite[Lemma 3.4]{HuXi08} that
$\End_{\cc_n}(T)$ and $\End_{\cc_n}(T')$ are derived equivalent.%

Putting the above arguments together, we obtain the desired result.
\end{proof}

\begin{theorem}\label{t:derived-equivalence}
Let $A$ and $A'$ be the endomorphism algebras of two maximal rigid
objects in $\cc_n$. Then $A$ and $A'$ are derived equivalent if and
only if their quivers have the same number of 3-cycles.
\end{theorem}
\begin{proof}
If the number of 3-cycles of the quivers are different, then by
Lemma~\ref{l:cartan-determinant} the determinants of the Cartan
matrices of the two algebras are not equal and thus they are not
derived equivalent, and the necessity follows.

It remains to prove the sufficiency. Let $\tilde{Q}$ be the quiver
of $A$, $c$ the vertex where the unique loop lies, and
$Q=\delta_c(\tilde{Q})$, the quiver obtained from $\tilde{Q}$ by
deleting the loop. By~\cite[Lemma 2.3]{BuanVatne08}, the quiver $Q$
can be obtained by a sequence of mutations without changing the
number of 3-cycles from the quiver
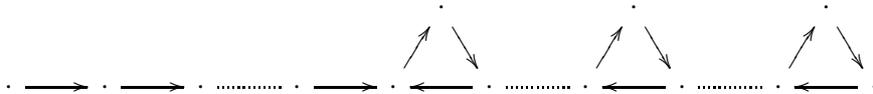
\[\xymatrix@!=0.5pc@C=1pc@R=2pc{&&&&&&&&&\cdot\ar[dr] &&&& \cdot\ar[dr] &&&& \cdot\ar[dr]\\
\cdot\ar[rr]&&\cdot\ar[rr]&&\cdot\ar@{.}[rr]&&\cdot\ar[rr]&&\cdot\ar[ur]&&\cdot\ar[ll]\ar@{.}[rr]&&\cdot\ar[ur]&&\cdot\ar[ll]\ar@{.}[rr]&&\cdot\ar[ur]&&\cdot\ar[ll]}\]
One proves by induction that in each intermediate quiver, the vertex
$c$ either has one neighbour, or has two neighbours and is traversed
by a 3-cycle. Moreover, if an intermediate mutation is performed at
$c$, then in the corresponding intermediate quiver $c$ must have
only one neighbour; otherwise the mutation changes the number of
3-cycles. Therefore by Lemma~\ref{l:mutation-and-loopization}, the
operation $\gamma_c$ commutes with the given sequence of mutations,
and hence $\tilde{Q}=\gamma_c(Q)$ can be obtained by a sequence of
mutations without changing the number of 3-cycles from the quiver
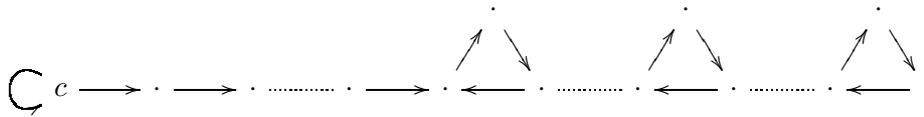
\[\xymatrix@!=0.5pc@C=1pc@R=2pc{&&&&&&&&&\cdot\ar[dr] &&&& \cdot\ar[dr] &&&& \cdot\ar[dr]\\
c\ar@(ul,dl)\ar[rr]&&\cdot\ar[rr]&&\cdot\ar@{.}[rr]&&\cdot\ar[rr]&&\cdot\ar[ur]&&\cdot\ar[ll]\ar@{.}[rr]&&\cdot\ar[ur]&&\cdot\ar[ll]\ar@{.}[rr]&&\cdot\ar[ur]&&\cdot\ar[ll]}\]
or the quiver
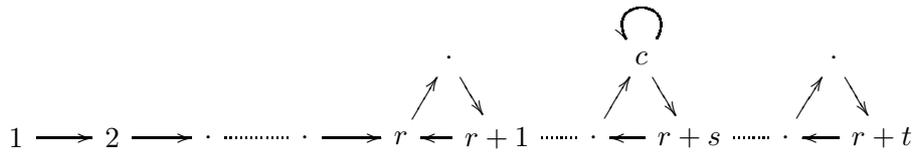
\[\xymatrix@!=0.5pc@C=1pc@R=2pc{&&&&&&&&&\cdot\ar[dr] &&&& c\ar@(ur,ul)\ar[dr] &&&& \cdot\ar[dr]\\
1\ar[rr]&&2\ar[rr]&&\cdot\ar@{.}[rr]&&\cdot\ar[rr]&& r
\ar[ur]&&r+1\ar[ll]\ar@{.}[rr]&&\cdot\ar[ur]&&r+s\ar[ll]\ar@{.}[rr]&&\cdot\ar[ur]&&r+t\ar[ll]}\]
One checks by induction that successive mutations (from the right to
the left) at the following sequence of vertices
\[(r,r-1,\ldots,2,c,2,\ldots,r-1,r;c,r+1,\ldots,r+s-1,c)\]
takes the second quiver to the first one. For example, when $r=3$,
$s=t=2$ ($n=8$), the sequence of mutations at vertices
$(3,2,c,2,3;c,4,c)$ yields the following sequence of quivers
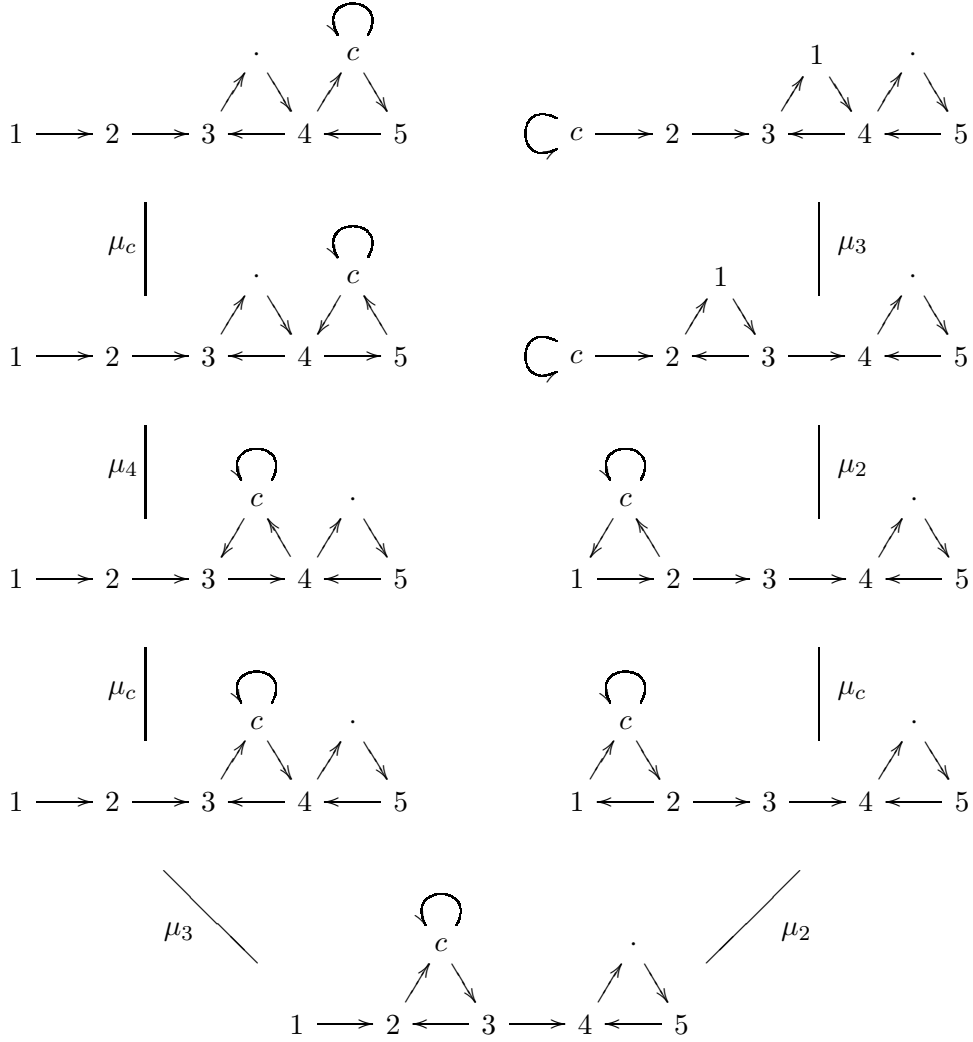
\[{\setlength{\unitlength}{0.7pt}
\begin{picture}(500,580)
\put(0,530){\xymatrix@!=0.5pc@C=1pc@R=2pc{&&&&&\cdot\ar[dr] && c\ar[dr]\ar@(ur,ul)\\
1\ar[rr]&&2\ar[rr]&&3 \ar[ur]&&4\ar[ur]\ar[ll]&&5\ar[ll]}}

\put(80,400){$\line(0,1){50}$} \put(60,425){$\mu_c$}

\put(0,410){\xymatrix@!=0.5pc@C=1pc@R=2pc{&&&&&\cdot\ar[dr] && c\ar[dl]\ar@(ur,ul)\\
1\ar[rr]&&2\ar[rr]&&3 \ar[ur]&&4\ar[ll]\ar[rr]&&5\ar[ul]}}

\put(80,280){$\line(0,1){50}$} \put(60,305){$\mu_4$}

\put(0,290){\xymatrix@!=0.5pc@C=1pc@R=2pc{&&&&&c\ar[dl]\ar@(ur,ul) && \cdot\ar[dr]\\
1\ar[rr]&&2\ar[rr]&&3 \ar[rr]&&4\ar[ul]\ar[ur]&&5\ar[ll]}}

\put(80,160){$\line(0,1){50}$} \put(60,185){$\mu_c$}

\put(0,170){\xymatrix@!=0.5pc@C=1pc@R=2pc{&&&&&c\ar[dr]\ar@(ur,ul) && \cdot\ar[dr] \\
1\ar[rr]&&2\ar[rr]&&3 \ar[ur]&&4\ar[ur]\ar[ll]&&5\ar[ll]}}

\put(90,90){$\line(1,-1){50}$} \put(90,55){$\mu_3$}

\put(150,50){\xymatrix@!=0.5pc@C=1pc@R=2pc{&&&c\ar[dr]\ar@(ur,ul)&& && \cdot\ar[dr] \\
1\ar[rr]&&2\ar[ur]&&3\ar[ll]\ar[rr] &&4\ar[ur]&&5\ar[ll]}}

\put(430,90){$\line(-1,-1){50}$} \put(420,55){$\mu_2$}

\put(300,170){\xymatrix@!=0.5pc@C=1pc@R=2pc{&c\ar[dr]\ar@(ur,ul)&&&& && \cdot\ar[dr] \\
1\ar[ur]&&2\ar[ll]\ar[rr]&&3\ar[rr] &&4\ar[ur]&&5\ar[ll]}}

\put(440,160){$\line(0,1){50}$} \put(450,185){$\mu_c$}

\put(300,290){\xymatrix@!=0.5pc@C=1pc@R=2pc{&c\ar@(ur,ul)\ar[dl]&&&& && \cdot\ar[dr] \\
1\ar[rr]&&2\ar[rr]\ar[ul]&&3\ar[rr] &&4\ar[ur]&&5\ar[ll]}}

\put(440,280){$\line(0,1){50}$} \put(450,305){$\mu_2$}

\put(300,410){\xymatrix@!=0.5pc@C=1pc@R=2pc{&&&1\ar[dr]&& && \cdot\ar[dr] \\
c\ar@(ul,dl)\ar[rr]&&2\ar[ur]&&3\ar[ll]\ar[rr]&&4\ar[ur]&&5\ar[ll]}}

\put(440,400){$\line(0,1){50}$} \put(450,425){$\mu_3$}

\put(300,530){\xymatrix@!=0.5pc@C=1pc@R=2pc{&&&&&1\ar[dr] && \cdot\ar[dr] \\
c\ar@(ul,dl)\ar[rr]&&2\ar[rr]&&3\ar[ur]&&4\ar[ur]\ar[ll]&&5\ar[ll]}}
\end{picture}
}\] One can check this example (and more examples) by using Keller's
quiver mutation applet~\cite{KellerQuiverMutationApplet}: one draws
the quiver without the loop and performs mutations, imagining that
there is a loop at the vertex $c$ and keeping in mind that each time
mutating at the vertex $c$ one has to add an arrow by hand (compare
$\mu_c\gamma_c(Q)$ and $\gamma_c\mu_c(Q)$ when $c$ is traversed by a
3-cycle).

Therefore the quiver of $A'$ can be obtained from the quiver of $A$
by a sequence $\underline{\mu}$ of mutations without changing the
number of 3-cycles. Suppose $A=\End_{\cc_n}(T)$ for a maximal rigid
object $T$, and let $T'=\underline{\mu}(T)$. Then by
Corollary~\ref{c:endo-algebra} and
Proposition~\ref{p:change-of-quiver}, we have $A'=\End_{\cc_n}(T')$.
So the sufficiency follows from
Proposition~\ref{p:derived-equivalence}.
\end{proof}

\begin{lemma}\label{l:cartan-determinant}
Let $A$ be the endomorphism algebra of a maximal rigid object of
$\cc_n$. Then the determinant of the Cartan matrix of $A$ is
$2^{t+1}$, where $t$ is the number of 3-cycles in the quiver of $A$.
\end{lemma}
\begin{proof}
The algebra $A$ is gentle, see~\cite[Theorem 3.1]{Vatne11}.
Moreover, all cycles are 3-cycles with full relations. Thus it
follows from~\cite[Theorem 1]{Holm05} that the determinant of the
Cartan matrix of $A$ is $2^{t+1}$.
\end{proof}

\section{Quivers with potential}\label{s:quiver-with-potential}

In this section we study the quivers with potential associated to
maximal rigid objects of cluster tubes. Assume the characteristic of
the base field $k$ is not $3$.

\subsection{Quivers with potential and their mutations} A \emph{quiver with potential} is a pair $(Q,W)$, where
$Q$ is a finite quiver and $W$ is an infinite linear combination of
nontrivial cycles of $Q$. To a quiver with potential is associated
an algebra, called the \emph{Jacobian algebra}, which is a certain
quotient of the completed path algebra of $Q$
(see~\cite{DerksenWeymanZelevinsky08} for the precise definition).
Assume that in the expression of $W$ all cycles have length $\geq
3$. Given a vertex $i$ of $Q$ which is not involved in a loop or
2-cycle, one can extend Derksen--Weyman--Zelevinsky's mutation in
~\cite{DerksenWeymanZelevinsky08} to $(Q,W)$ at $i$ (the quivers in
~\cite{DerksenWeymanZelevinsky08} do not have loops). The mutation
yields a new quiver with potential, denoted by $\mu_i(Q,W)$.

\subsection{Quivers with potential in cluster
tubes}\label{s:change-of-qp} Fix an integer $n\geq 2$ and let
$\cc_n$ be the cluster tube of rank $n$.

Let $T$ be a basic maximal rigid object in $\cc_n$. We associate to
$T$ a quiver with potential $(\tilde{Q},\tilde{W})$ as follows. By
Corollary~\ref{c:endo-algebra} the endomorphism algebra
$A=\End_{\cc_n}(T)$ of $T$ in $\cc_n$ is isomorphic to
$k\tilde{Q}/I_{\tilde{Q}}$, where $\tilde{Q}$ is a quiver in the set
$\widetilde{\cq}_{n-1}$. Let $\tilde{W}=W_{\tilde{Q}}$ be the sum of
the cube of the unique loop and all 3-cycles of $\tilde{Q}$. It
follows from the definition of $I_{\tilde{Q}}$ that $A$ is
isomorphic to the Jacobian algebra of $(\tilde{Q},\tilde{W})$.

The category $\cc_n$ being Hom-finite, the algebra $A$ is a
finite-dimensional Jacobian algebra. Therefore it follows
from~\cite[Theorem 3.6]{Amiot09} that there is a 2-Calabi--Yau
triangulated category $\cc$ with a cluster-tilting object $M$ whose
endomorphism algebra is isomorphic to $A$.
Applying~\cite[Proposition 2.1, Theorem 3.3]{KellerReiten07} to the
pair $(\cc,M)$, we obtain that $A$ is Gorenstein of Gorenstein
dimension at most 1 and is stably 3-Calabi--Yau.

Let $c$ be the vertex of $\tilde{Q}$ at which the unique loop lies.
Let $Q$ be the quiver obtained from $\tilde{Q}$ by deleting the
loop, and $W$ the potential obtained from $\tilde{W}$ by subtracting
the cube of the loop. Then the Jacobian algebra of $(Q,W)$ is
isomorphic to the cluster-tilted algebra $B$ of $T$ viewed as a
titling module in $\rep\overrightarrow{A}_{n-1}$ (see
Theorem~\ref{t:endo-algebra}). Denote
$(\tilde{Q},\tilde{W})=\gamma_c(Q,W)$,
$(Q,W)=\delta_c(\tilde{Q},\tilde{W})$.

\begin{lemma}\label{l:mutation-and-loopization-qp} For a vertex $i$ of $Q$ different from $c$, we have
$\mu_i\gamma_c(Q,W)=\gamma_c\mu_i(Q,W)$.
\end{lemma}
\begin{proof} This follows from the definitions of $\mu_i$ and
$\gamma_c$.
\end{proof}

Let $\widetilde{\cq\cp}_{n-1}$ denote the set of quivers with
potential $(\tilde{Q},W_{\tilde{Q}})$ for
$\tilde{Q}\in\widetilde{\cq}_{n-1}$. Corollary~\ref{c:endo-algebra}
is reformulated as

\begin{corollary}\label{c:endo-algebra-qp}
An algebra is the endomorphism algebra of a maximal rigid object of
$\cc_n$ if and only if it is isomorphic to the Jacobian algebra of a
quiver with potential in $\widetilde{\cq\cp}_{n-1}$.
\end{corollary}

Next we study the change of the quivers with potential in
$\widetilde{\cq\cp}_{n-1}$ induced from the mutation of maximal
rigid objects. We first give some examples.

\begin{example} In this example we draw the exchange graph of mutations of maximal rigid objects in
$\cc_n$ for $n=2,3,4$ and the corresponding graph with maximal rigid
objects replaced by quivers with potential associated to them. In
the former graph, (indecomposable direct summands of) maximal rigid
objects are given as filled-in circles of the first $n-1$ layers of
the Auslander--Reiten quiver of $\cc_n$.

\bigskip
n=2: There are two basic maximal rigid objects up to isomorphism,
which are indecomposable. The mutation graphs are shown in
Figures~\ref{f:2.1} and \ref{f:2.2}.\begin{figure}
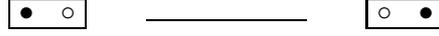
\[
\framebox[6\totalheight]{$\bullet\hspace{10pt}\circ$}\qquad{\line(1,0){60}}\qquad\framebox[6\totalheight]{$\circ\hspace{10pt}\bullet$}
\]\caption{The mutation graph of $\cc_2$ -- maximal rigid
objects}\label{f:2.1}\end{figure} \begin{figure}
\[ (\xymatrix{\cdot \ar@(dl,ul)^{\varphi}},\varphi^{3})\qquad {\line(1,0){60}}\qquad (\xymatrix{\cdot \ar@(dl,ul)^{\varphi}},\varphi^{3}) \]
\caption{The mutation graph of $\cc_2$ -- quivers with potential}
\label{f:2.2}\end{figure}
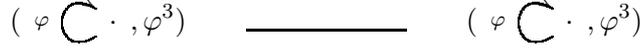

\bigskip
n=3: There are six basic maximal rigid objects up to isomorphism,
each of which has two indecomposable direct summands. The mutation
graphs are shown in Figures~\ref{f:3.1} and \ref{f:3.2}.
\begin{figure}
\[{\setlength{\unitlength}{0.7pt}
\begin{picture}(500,320)
\put(200,10){\framebox[85pt]{\xymatrix@C=0.15pt@R=0.15pt{&\circ&&\bullet&&\circ\\
\circ&&\circ&&\bullet}}}
\put(365,100){\framebox[85pt]{\xymatrix@C=0.15pt@R=0.15pt{&\circ&&\circ&&\bullet\\
\circ&&\circ&&\bullet}}}
\put(365,200){\framebox[85pt]{\xymatrix@C=0.15pt@R=0.15pt{&\circ&&\circ&&\bullet\\
\bullet&&\circ&&\circ}}}
\put(40,200){\framebox[85pt]{\xymatrix@C=0.15pt@R=0.15pt{&\bullet&&\circ&&\circ\\
\circ&&\bullet&&\circ}}}
\put(40,100){\framebox[85pt]{\xymatrix@C=0.15pt@R=0.15pt{&\circ&&\bullet&&\circ\\
\circ&&\bullet&&\circ}}}
\put(200,300){\framebox[85pt]{\xymatrix@C=0.15pt@R=0.15pt{&\bullet&&\circ&&\circ\\
\bullet&&\circ&&\circ}}}

\put(110,60){$\line(5,-4){80}$} \put(410,60){$\line(-5,-4){80}$}
 \put(100,125){$\line(0,1){30}$}
 \put(430,125){$\line(0,1){30}$}
 \put(110,220){$\line(5,4){80}$}
 \put(410,220){$\line(-5,4){80}$}
 \end{picture}}
\]
\caption{The mutation graph of $\cc_3$ -- maximal rigid
objects}\label{f:3.1}
\end{figure}
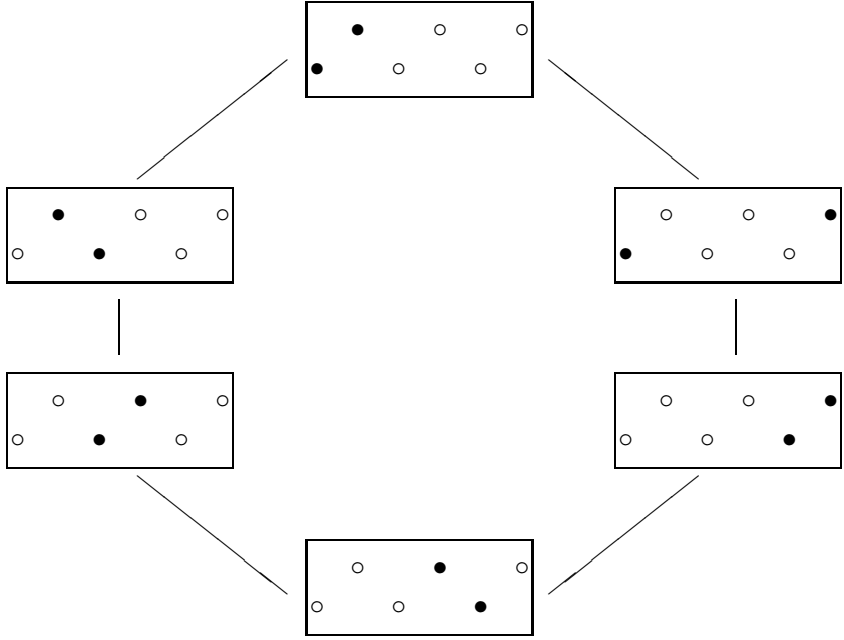
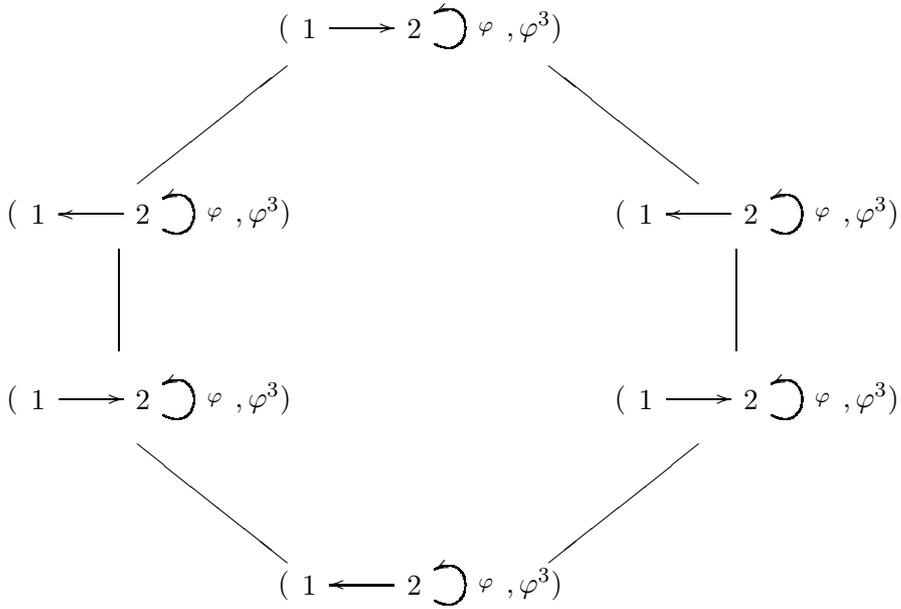
\begin{figure}
\[{\setlength{\unitlength}{0.7pt}
\begin{picture}(500,320)
\put(185,0){$(\xymatrix{1&2\ar[l]\ar@(dr,ur)_{\varphi}},\varphi^3)$}
\put(365,100){$(\xymatrix{1\ar[r]&2\ar@(dr,ur)_{\varphi}},\varphi^3)$}
\put(365,200){$(\xymatrix{1&2\ar[l]\ar@(dr,ur)_{\varphi}},\varphi^3)$}
\put(40,200){$(\xymatrix{1&2\ar[l]\ar@(dr,ur)_{\varphi}},\varphi^3)$}
\put(40,100){$(\xymatrix{1\ar[r]&2\ar@(dr,ur)_{\varphi}},\varphi^3)$}
\put(185,300){$(\xymatrix{1\ar[r]&2\ar@(dr,ur)_{\varphi}},\varphi^3)$}

\put(110,80){$\line(5,-4){80}$}
\put(410,80){$\line(-5,-4){80}$}
 \put(100,130){$\line(0,1){55}$}
 \put(430,130){$\line(0,1){55}$}
 \put(110,220){$\line(5,4){80}$}
 \put(410,220){$\line(-5,4){80}$}
 \end{picture}}
\]
\caption{The mutation graph of $\cc_3$ -- quivers with
potential}\label{f:3.2}
\end{figure}

\bigskip
n=4: Up to isomorphism there are twenty basic maximal rigid objects,
each of which has three indecomposable direct summands.
Figure~\ref{f:4.1} is the mutation graph for maximal rigid objects.
In the graph, there are four pentagons, each of which corresponds to
an integer $a=1,2,3,4$: the vertices of the pentagon corresponding
to $a$ are maximal rigid objects in the wing of $(a,3)$. In
Figure~\ref{f:4.2} we only give the mutation graph of the quivers
(the potentials are uniquely determined by the quivers), where the
wavy lines mean that there is a reordering of the vertices besides
the mutation. \footnote{Many thanks to Laurent Demonet for much help
in drawing these two graphs.}

\newcommand{\sbullet}{\scriptstyle\bullet}
\newcommand{\scirc}{\scriptstyle\circ}
\newcommand{\MA}{\begin{array}{*{8}{c@{\hspace{-1pt}}}}
\scirc&&\scirc&&\scirc&&\sbullet&\\[-13pt]
&\scirc&&\scirc&&\scirc&&\sbullet\\[-13pt]
\sbullet&&\scirc&&\scirc&&\scirc&
\end{array}}
\newcommand{\MB}{\begin{array}{*{8}{c@{\hspace{-1pt}}}}
\sbullet&&\scirc&&\scirc&&\scirc&\\[-13pt]
&\scirc&&\scirc&&\scirc&&\sbullet\\[-13pt]
\sbullet&&\scirc&&\scirc&&\scirc&
\end{array}}
\newcommand{\MC}{\begin{array}{*{8}{c@{\hspace{-1pt}}}}
\sbullet&&\scirc&&\scirc&&\scirc&\\[-13pt]
&\scirc&&\scirc&&\scirc&&\sbullet\\[-13pt]
\scirc&&\scirc&&\scirc&&\sbullet&
\end{array}}
\newcommand{\MD}{\begin{array}{*{8}{c@{\hspace{-1pt}}}}
\scirc&&\scirc&&\scirc&&\sbullet&\\[-13pt]
&\scirc&&\scirc&&\scirc&&\sbullet\\[-13pt]
\scirc&&\scirc&&\scirc&&\sbullet&
\end{array}}
\newcommand{\ME}{\begin{array}{*{8}{c@{\hspace{-1pt}}}}
\scirc&&\scirc&&\scirc&&\sbullet&\\[-13pt]
&\scirc&&\scirc&&\scirc&&\scirc\\[-13pt]
\sbullet&&\scirc&&\sbullet&&\scirc&
\end{array}}
\newcommand{\MF}{\begin{array}{*{8}{c@{\hspace{-1pt}}}}
\sbullet&&\scirc&&\scirc&&\scirc&\\[-13pt]
&\sbullet&&\scirc&&\scirc&&\scirc\\[-13pt]
\sbullet&&\scirc&&\scirc&&\scirc&
\end{array}}
\newcommand{\MG}{\begin{array}{*{8}{c@{\hspace{-1pt}}}}
\sbullet&&\scirc&&\scirc&&\scirc&\\[-13pt]
&\scirc&&\scirc&&\scirc&&\scirc\\[-13pt]
\scirc&&\sbullet&&\scirc&&\sbullet&
\end{array}}
\newcommand{\MH}{\begin{array}{*{8}{c@{\hspace{-1pt}}}}
\scirc&&\scirc&&\scirc&&\sbullet&\\[-13pt]
&\scirc&&\scirc&&\sbullet&&\scirc\\[-13pt]
\scirc&&\scirc&&\scirc&&\sbullet&
\end{array}}
\newcommand{\MI}{\begin{array}{*{8}{c@{\hspace{-1pt}}}}
\scirc&&\scirc&&\scirc&&\sbullet&\\[-13pt]
&\scirc&&\scirc&&\sbullet&&\scirc\\[-13pt]
\scirc&&\scirc&&\sbullet&&\scirc&
\end{array}}
\newcommand{\MJ}{\begin{array}{*{8}{c@{\hspace{-1pt}}}}
\scirc&&\sbullet&&\scirc&&\scirc&\\[-13pt]
&\scirc&&\scirc&&\scirc&&\scirc\\[-13pt]
\sbullet&&\scirc&&\sbullet&&\scirc&
\end{array}}

\newcommand{\MK}{\begin{array}{*{8}{c@{\hspace{-1pt}}}}
\scirc&&\sbullet&&\scirc&&\scirc&\\[-13pt]
&\sbullet&&\scirc&&\scirc&&\scirc\\[-13pt]
\sbullet&&\scirc&&\scirc&&\scirc&
\end{array}}

\newcommand{\ML}{\begin{array}{*{8}{c@{\hspace{-1pt}}}}
\sbullet&&\scirc&&\scirc&&\scirc&\\[-13pt]
&\sbullet&&\scirc&&\scirc&&\scirc\\[-13pt]
\scirc&&\sbullet&&\scirc&&\scirc&
\end{array}}

\newcommand{\MM}{\begin{array}{*{8}{c@{\hspace{-1pt}}}}
\scirc&&\scirc&&\sbullet&&\scirc&\\[-13pt]
&\scirc&&\scirc&&\scirc&&\scirc\\[-13pt]
\scirc&&\sbullet&&\scirc&&\sbullet&
\end{array}}

\newcommand{\MN}{\begin{array}{*{8}{c@{\hspace{-1pt}}}}
\scirc&&\scirc&&\sbullet&&\scirc&\\[-13pt]
&\scirc&&\scirc&&\sbullet&&\scirc\\[-13pt]
\scirc&&\scirc&&\scirc&&\sbullet&
\end{array}}
\newcommand{\MO}{\begin{array}{*{8}{c@{\hspace{-1pt}}}}
\scirc&&\scirc&&\sbullet&&\scirc&\\[-13pt]
&\scirc&&\scirc&&\sbullet&&\scirc\\[-13pt]
\scirc&&\scirc&&\sbullet&&\scirc&
\end{array}}
\newcommand{\MP}{\begin{array}{*{8}{c@{\hspace{-1pt}}}}
\scirc&&\scirc&&\sbullet&&\scirc&\\[-13pt]
&\scirc&&\sbullet&&\scirc&&\scirc\\[-13pt]
\scirc&&\scirc&&\sbullet&&\scirc&
\end{array}}

\newcommand{\MQ}{\begin{array}{*{8}{c@{\hspace{-1pt}}}}
\scirc&&\sbullet&&\scirc&&\scirc&\\[-13pt]
&\scirc&&\sbullet&&\scirc&&\scirc\\[-13pt]
\scirc&&\scirc&&\sbullet&&\scirc&
\end{array}}

\newcommand{\MU}{\begin{array}{*{8}{c@{\hspace{-1pt}}}}
\scirc&&\sbullet&&\scirc&&\scirc&\\[-13pt]
&\sbullet&&\scirc&&\scirc&&\scirc\\[-13pt]
\scirc&&\sbullet&&\scirc&&\scirc&
\end{array}}
\newcommand{\MS}{\begin{array}{*{8}{c@{\hspace{-1pt}}}}
\scirc&&\sbullet&&\scirc&&\scirc&\\[-13pt]
&\scirc&&\sbullet&&\scirc&&\scirc\\[-13pt]
\scirc&&\sbullet&&\scirc&&\scirc&
\end{array}}

\newcommand{\MT}{\begin{array}{*{8}{c@{\hspace{-1pt}}}}
\scirc&&\scirc&&\sbullet&&\scirc&\\[-13pt]
&\scirc&&\sbullet&&\scirc&&\scirc\\[-13pt]
\scirc&&\sbullet&&\scirc&&\scirc&
\end{array}}
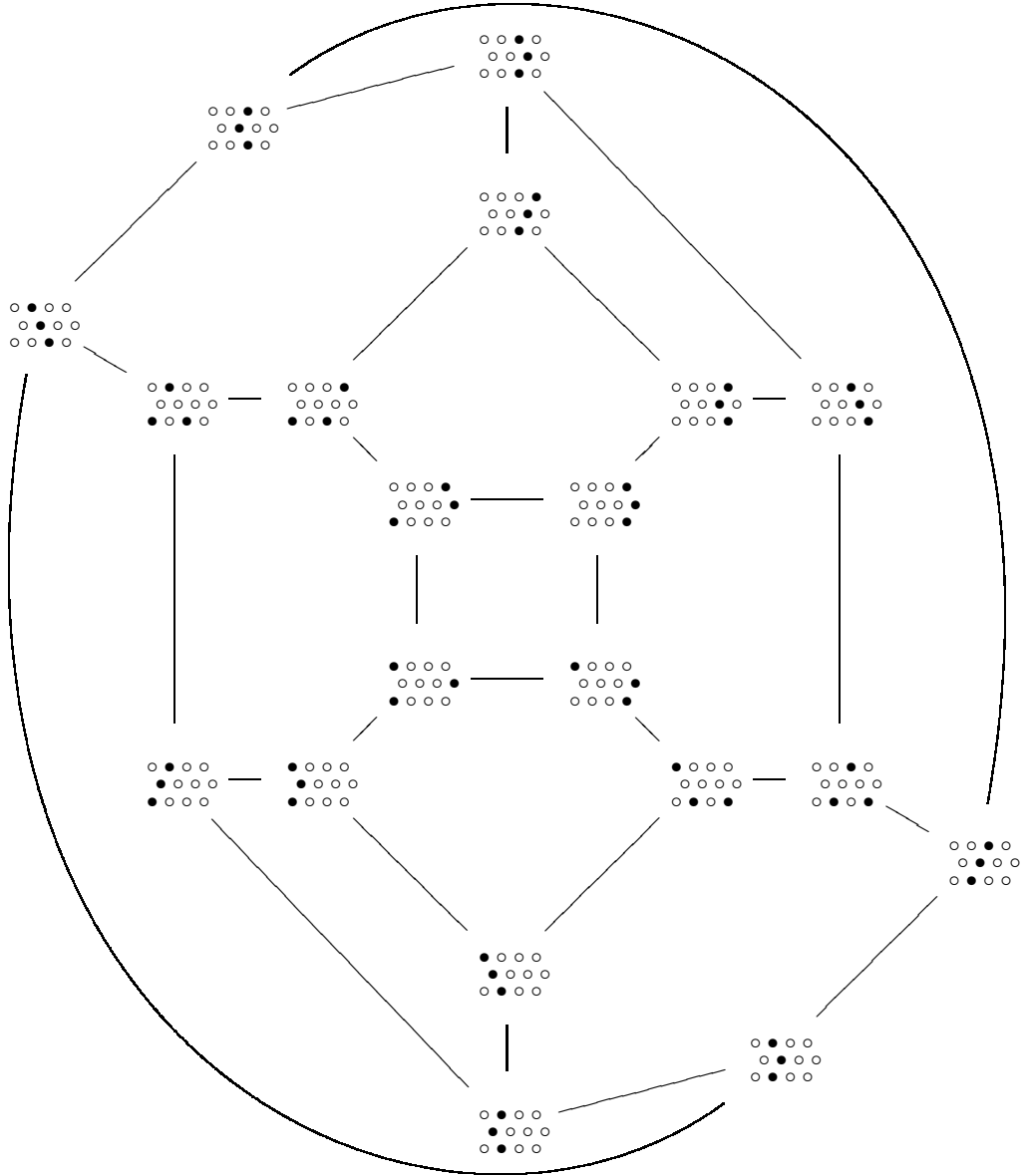
\begin{figure}
 {
 $$\begin{xy}
            0;<0pc,4pc>:,="D",
        {\xypolygon4"A"{~*{\ifcase\xypolynode
        \or \MA \or \MB  \or \MC  \or \MD  \fi }~:{(1,0):}~>{}}},
        {\xypolygon4"B"{~*{\ifcase\xypolynode\or \ME  \or \MF \or \MG \or \MH \fi}~:{(2.12132,0):}~>{}}}, 
            {\xypolygon6"C"{~*{\ifcase\xypolynode\or \MI \or \MJ \or \MK \or \ML \or \MM \or \MN \fi}~:{(3,0):}~>{}}},
        {\xypolygon12"D"{~*{\ifcase\xypolynode\or \MO \or \MP \or \MQ \or \or \or \or \MU \or \MS \or \MT \or \or \or \fi}~:{(4.098,-1.098):}~>{}}},
            "A1";"A2"**\dir{-},
        "A2";"A3"**\dir{-},
            "A3";"A4"**\dir{-},
        "A4";"A1"**\dir{-},
        "A1";"B1"**\dir{-},
        "A2";"B2"**\dir{-},
        "A3";"B3"**\dir{-},
        "A4";"B4"**\dir{-},
        "B1";"C1"**\dir{-},
        "B1";"C2"**\dir{-},
        "B2";"C3"**\dir{-},
        "B2";"C4"**\dir{-},
        "B3";"C4"**\dir{-},
        "B3";"C5"**\dir{-},
        "B4";"C6"**\dir{-},
        "B4";"C1"**\dir{-},
        "C2";"C3"**\dir{-},
        "C5";"C6"**\dir{-},
        "D1";"C1"**\dir{-},
        "D1";"C6"**\dir{-},
        "D7";"C4"**\dir{-},
        "D7";"C3"**\dir{-},
        "D1";"D2"**\dir{-},
        "D2";"D3"**\dir{-},
        "C2";"D3"**\dir{-},
        "D7";"D8"**\dir{-},
        "D8";"D9"**\dir{-},
        "C5";"D9"**\dir{-},
        "D3";"D8"**\crv{(-3,5)&(-5,0)},
        "D9";"D2"**\crv{(5,-5)&(7,0)},
    \end{xy}$$
 }\caption{The mutation graph of $\cc_4$ -- maximal rigid objects}\label{f:4.1}
\end{figure}

\newcommand{\QA}{\begin{xy}
\xymatrix@R=1pc@C=0.577pc@M=0pc@W=0.25pc@H=0.25pc{&\cdot\ar@`{p+(0.09,0.18),p+(0.36,0),p+(0.09,-0.18)}[]\ar[dl]&\\
\cdot\ar[rr]&&\cdot}\end{xy}}

\newcommand{\QB}{\begin{xy}
\xymatrix@R=1pc@C=0.577pc@M=0pc@W=0.25pc@H=0.25pc{\\&\cdot\ar@`{p+(0.09,0.18),p+(0.36,0),p+(0.09,-0.18)}[]&\\
\cdot\ar[rr]\ar[ur]&&\cdot}\end{xy}}

\newcommand{\QC}{\begin{xy}
\xymatrix@R=1pc@C=0.577pc@M=0pc@W=0.25pc@H=0.25pc{\\&\cdot\ar@`{p+(0.09,0.18),p+(0.36,0),p+(0.09,-0.18)}[]&\\
\cdot\ar[ur]&&\cdot\ar[ll]}\end{xy}}

\newcommand{\QD}{\begin{xy}
\xymatrix@R=1pc@C=0.577pc@M=0pc@W=0.25pc@H=0.25pc{&\cdot\ar@`{p+(0.09,0.18),p+(0.36,0),p+(0.09,-0.18)}[]\ar[dl]&\\
\cdot&&\cdot\ar[ll]}\end{xy}}

\newcommand{\QE}{\begin{xy}
\xymatrix@R=1pc@C=0.577pc@M=0pc@W=0.25pc@H=0.25pc{&\cdot\ar@`{p+(0.045,0.09),p+(0.18,0),p+(0.045,-0.09)}[]\ar[dr]&\\
\cdot\ar[ur]&&\cdot\ar[ll]}\end{xy}}

\newcommand{\QF}{\begin{xy}
\xymatrix@R=1pc@C=0.577pc@M=0pc@W=0.25pc@H=0.25pc{&\cdot\ar@`{p+(0.045,0.09),p+(0.18,0),p+(0.045,-0.09)}[]\ar[dl]&\\
\cdot&&\cdot\ar[ll]}\end{xy}}

\newcommand{\QG}{\begin{xy}
\xymatrix@R=1pc@C=0.577pc@M=0pc@W=0.25pc@H=0.25pc{&\cdot\ar@`{p+(0.045,0.09),p+(0.18,0),p+(0.045,-0.09)}[]\ar[dl]&\\
\cdot\ar[rr]&&\cdot\ar[ul]}\end{xy}}

\newcommand{\QH}{\begin{xy}
\xymatrix@R=1pc@C=0.577pc@M=0pc@W=0.25pc@H=0.25pc{\\&\cdot\ar@`{p+(0.045,0.09),p+(0.18,0),p+(0.045,-0.09)}[]&\\
\cdot\ar[rr]\ar[ur]&&\cdot}\end{xy}}

\newcommand{\QI}{\begin{xy}
\xymatrix@R=1pc@C=0.577pc@M=0pc@W=0.25pc@H=0.25pc{\\&\cdot\ar@`{p+(0.03,0.06),p+(0.12,0),p+(0.03,-0.06)}[]&\\
\cdot\ar[rr]&&\cdot\ar[ul]}\end{xy}}

\newcommand{\QJ}{\begin{xy}
\xymatrix@R=1pc@C=0.577pc@M=0pc@W=0.25pc@H=0.25pc{\\&\cdot\ar@`{p+(0.03,0.06),p+(0.12,0),p+(0.03,-0.06)}[]\ar[dl]&\\
\cdot\ar[rr]&&\cdot\ar[ul]}\end{xy}}

\newcommand{\QK}{\begin{xy}
\xymatrix@R=1pc@C=0.577pc@M=0pc@W=0.25pc@H=0.25pc{\\&\cdot\ar@`{p+(0.03,0.06),p+(0.12,0),p+(0.03,-0.06)}[]&\\
\cdot\ar[ur]&&\cdot\ar[ll]}\end{xy}}

\newcommand{\QL}{\begin{xy}
\xymatrix@R=1pc@C=0.577pc@M=0pc@W=0.25pc@H=0.25pc{&\cdot\ar@`{p+(0.03,0.06),p+(0.12,0),p+(0.03,-0.06)}[]\ar[dr]&\\
\cdot&&\cdot\ar[ll]}\end{xy}}

\newcommand{\QM}{\begin{xy}
\xymatrix@R=1pc@C=0.577pc@M=0pc@W=0.25pc@H=0.25pc{\\&\cdot\ar@`{p+(0.03,0.06),p+(0.12,0),p+(0.03,-0.06)}[]\ar[dr]&\\
\cdot\ar[ur]&&\cdot\ar[ll]}\end{xy}}

\newcommand{\QN}{\begin{xy}
\xymatrix@R=1pc@C=0.577pc@M=0pc@W=0.25pc@H=0.25pc{&\cdot\ar@`{p+(0.03,0.06),p+(0.12,0),p+(0.03,-0.06)}[]\ar[dl]&\\
\cdot\ar[rr]&&\cdot}\end{xy}}

\newcommand{\QO}{\begin{xy}
\xymatrix@R=1pc@C=0.577pc@M=0pc@W=0.25pc@H=0.25pc{\\&\cdot\ar@`{p+(0.02,0.04),p+(0.08,0),p+(0.02,-0.04)}[]\ar[dr]&\\
\cdot\ar[rr]&&\cdot}\end{xy}}

\newcommand{\QP}{\begin{xy}
\xymatrix@R=1pc@C=0.577pc@M=0pc@W=0.25pc@H=0.25pc{&\cdot\ar@`{p+(0.02,0.04),p+(0.08,0),p+(0.02,-0.04)}[]&\\
\cdot&&\cdot\ar[ll]\ar[ul]}\end{xy}}

\newcommand{\QQ}{\begin{xy}
\xymatrix@R=1pc@C=0.577pc@M=0pc@W=0.25pc@H=0.25pc{&\cdot\ar@`{p+(0.02,0.04),p+(0.08,0),p+(0.02,-0.04)}[]\ar[dr]&\\
\cdot&&\cdot\ar[ll]}\end{xy}}

\newcommand{\QR}{\begin{xy}
\xymatrix@R=1pc@C=0.577pc@M=0pc@W=0.25pc@H=0.25pc{\\&\cdot\ar@`{p+(0.02,0.04),p+(0.08,0),p+(0.02,-0.04)}[]&\\
\cdot&&\cdot\ar[ll]\ar[ul]\\}\end{xy}}

\newcommand{\QS}{\begin{xy}
\xymatrix@R=1pc@C=0.577pc@M=0pc@W=0.25pc@H=0.25pc{&\cdot\ar@`{p+(0.02,0.04),p+(0.08,0),p+(0.02,-0.04)}[]\ar[dr]&\\
\cdot\ar[rr]&&\cdot}\end{xy}}

\newcommand{\QT}{\begin{xy}
\xymatrix@R=1pc@C=0.577pc@M=0pc@W=0.25pc@H=0.25pc{\\&\cdot\ar@`{p+(0.02,0.04),p+(0.08,0),p+(0.02,-0.04)}[]&\\
\cdot\ar[rr]&&\cdot\ar[ul]}\end{xy}}

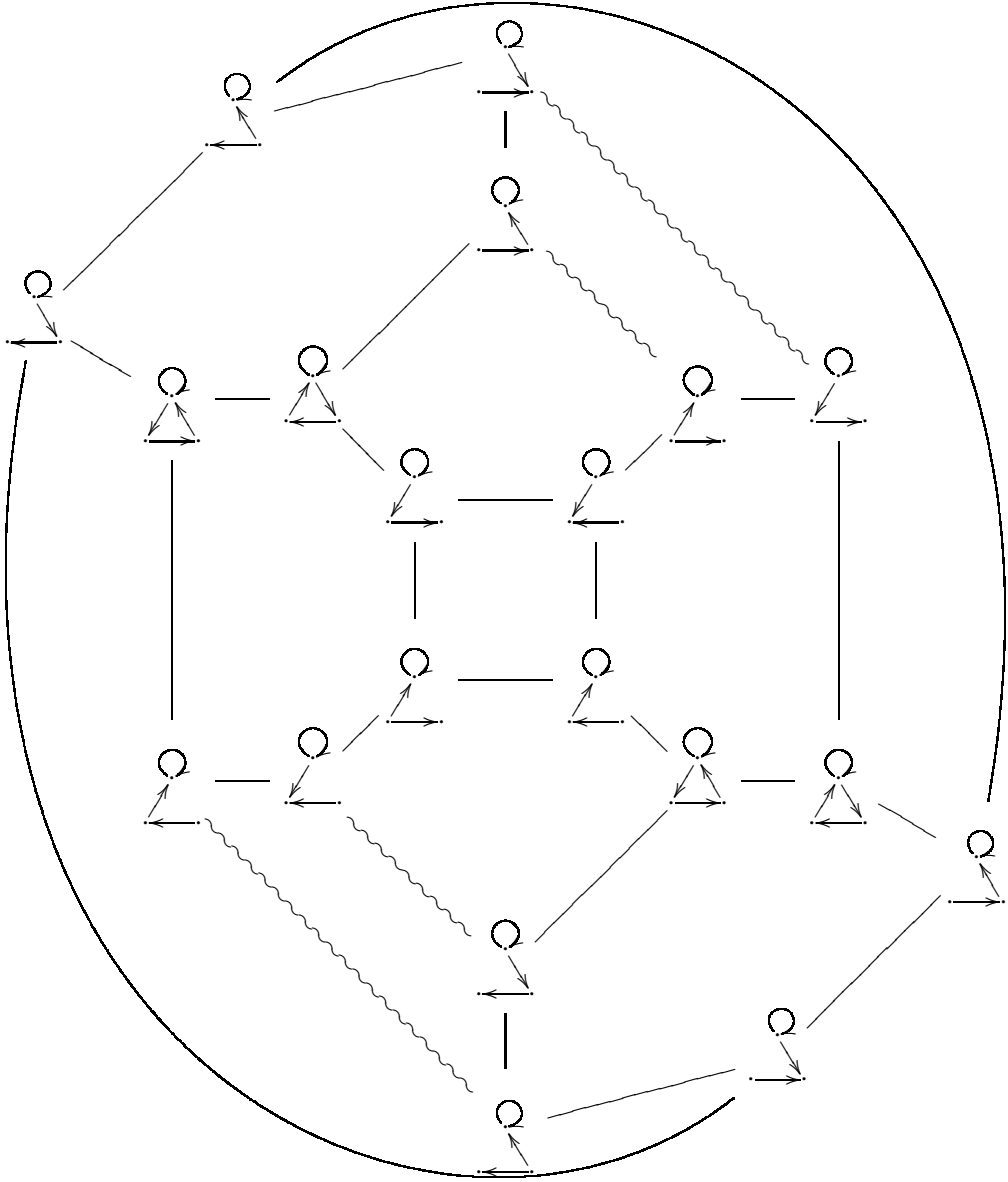
\begin{figure}
 {
 $$\begin{xy}
            0;<0pc,4pc>:,="D",
        {\xypolygon4"A"{~*{\ifcase\xypolynode\or \QA \or \QB \or \QC \or \QD \fi }~:{(1,0):}~>{}}},
        {\xypolygon4"B"{~*{\ifcase\xypolynode\or \QE \or \QF \or \QG \or \QH \fi}~:{(2.12132,0):}~>{}}}, 
            {\xypolygon6"C"{~*{\ifcase\xypolynode\or \QI \or \QJ  \or \QK \or \QL \or \QM \or \QN \fi}~:{(3,0):}~>{}}},
        {\xypolygon12"D"{~*{\ifcase\xypolynode\or \QO \or \QP \or \QQ \or \or \or \or \QR \or \QS \or \QT \or \or \or \fi}~:{(4.098,-1.098):}~>{}}},
            "A1";"A2"**\dir{-},
        "A2";"A3"**\dir{-},
            "A3";"A4"**\dir{-},
        "A4";"A1"**\dir{-},
        "A1";"B1"**\dir{-},
        "A2";"B2"**\dir{-},
        "A3";"B3"**\dir{-},
        "A4";"B4"**\dir{-},
        "B1";"C1"**\dir{-},
        "B1";"C2"**\dir{-},
        "B2";"C3"**\dir{-},
        "B2";"C4"**\dir{~},
        "B3";"C4"**\dir{-},
        "B3";"C5"**\dir{-},
        "B4";"C6"**\dir{-},
        "B4";"C1"**\dir{~},
        "C2";"C3"**\dir{-},
        "C5";"C6"**\dir{-},
        "D1";"C1"**\dir{-},
        "D1";"C6"**\dir{~},
        "D7";"C4"**\dir{-},
        "D7";"C3"**\dir{~},
        "D1";"D2"**\dir{-},
        "D2";"D3"**\dir{-},
        "C2";"D3"**\dir{-},
        "D7";"D8"**\dir{-},
        "D8";"D9"**\dir{-},
        "C5";"D9"**\dir{-},
        "D3";"D8"**\crv{(-3,5)&(-5,0)},
        "D9";"D2"**\crv{(5,-5)&(7,0)},
    \end{xy}$$
 }
 \caption{The mutation graph of $\cc_4$ -- quivers}
\label{f:4.2}
\end{figure}
\end{example}

Generally we have a quiver with potential version of
Proposition~\ref{p:change-of-quiver}.

\begin{proposition}\label{p:change-of-qp} Assume $n\geq 3$. Let $T$ and $T'$ be two maximal rigid objects of
$\cc_n$ related by a mutation. Let $(\tilde{Q},\tilde{W})$ and
$(\tilde{Q}',\tilde{W}')$ be the quivers with potential associated
to $T$ and $T'$ respectively. Assume that $T$ is in the wing of
$(a,n-1)$.
\begin{itemize}
\item[a)] If the mutation is simple, then $(\tilde{Q},\tilde{W})$ and
$(\tilde{Q}',\tilde{W}')$ are related by a
Derksen--Weyman--Zelevinsky mutation.
\item[b)] If the mutation is not simple, there are two cases
\begin{itemize}
\item[1)] if the vertex $c$ has one neighbour, then $\tilde{Q'}$
is obtained from $\tilde{Q}$ by reversing the unique arrow adjacent
to $c$ and $\tilde{W'}=\tilde{W}$;
\item[2)] if the vertex $c$ has two neighbours, then $\tilde{Q'}$
can be obtained from $\tilde{Q}$ by reversing all arrows in the
unique 3-cycle $C$ traversing $c$ and $\tilde{W'}$ is obtained from
$\tilde{W}$ by replacing $C$ by the new 3-cycle. 
\end{itemize}
\end{itemize}

\end{proposition}
\begin{proof}
a) Suppose that $T$ and $T'$ are related by a simple mutation. Then
by Lemma~\ref{l:simple-mutation} they are in the wing of the same
$(a,n-1)$ for some $a=1,\ldots,n$. It follows from
Lemma~\ref{l:mutation-tilting-module} that $T$ and $T'$, viewed as
tilting modules in $\rep\overrightarrow{A}_{n-1}$, are related by a
mutation. It follows that the quivers with potential $(Q,W)$ and
$(Q',W')$ associated with the corresponding cluster-tilted algebras
are related by a Derksen--Weyman--Zelevinsky mutation,
\confer~\cite[Theorem 5.1]{BuanIyamaReitenSmith08}, which is
performed at a vertex $i\neq c$. By
Lemma~\ref{l:mutation-and-loopization}, we have
\[(\tilde{Q}',\tilde{W}')=\gamma_c(Q',W')=\gamma_c\mu_i(Q,W)=\mu_i\gamma_c(Q,W)=\mu_i(\tilde{Q},\tilde{W}),\]
finishing the proof of a).

b) The assertion follows from Proposition~\ref{p:change-of-quiver}
b) and Corollary~\ref{c:endo-algebra-qp}.
\end{proof}

As in the quiver case, the quiver with potential
$(\tilde{Q}',\tilde{W}')$ in Proposition~\ref{p:change-of-qp} only
depends on the quiver with potential
 $(\tilde{Q},\tilde{W})$  and the
vertex $i$ at which the mutation is taken, and does not depend on
the choice of the maximal rigid object $T$. We write
$(\tilde{Q}',\tilde{W}')=\mu_i(\tilde{Q},\tilde{W})$.



We reformulate Proposition~\ref{p:derived-equivalence} and
Theorem~\ref{t:derived-equivalence} in terms of quivers with
potential.
\begin{theorem}
Let $(\tilde{Q},\tilde{W})$ and $(\tilde{Q}',\tilde{W}')$ be two
quivers with potential in $\widetilde{\cq\cp}_{n-1}$. Then their
Jacobian algebras are derived equivalent if and only if the quivers
$\tilde{Q}$ and $\tilde{Q}'$ have the same number of 3-cycles. In
particular, for a vertex $i$ of $\tilde{Q}$, the Jacobian algebras
of $(\tilde{Q},\tilde{W})$ and $\mu_i(\tilde{Q},\tilde{W})$ are
derived equivalent if and only if the mutation does not change the
number of 3-cycles of the quiver.
\end{theorem}


\begin{thebibliography}{10}


\bibitem{Amiot09}
   Claire Amiot, \emph{Cluster categories for algebras of global dimension $2$ and
quivers with potential}, {Annales de l'institut Fourier} \textbf{59}
(2009), no.~{6}, {2525--2590}.

\bibitem{BuanIyamaReitenScott09}
Aslak Bakke~Buan, Osamu Iyama, Idun Reiten, and Jeanne
        Scott, \emph{Cluster structures for 2-{C}alabi-{Y}au categories and unipotent groups},
 {Compos. Math.} \textbf{145} (2009), no.~{4}, {1035--1079}.

\bibitem{BuanIyamaReitenSmith08}
Aslak Bakke~Buan, Osamu Iyama, Idun Reiten, and David Smith,
\emph{Mutation of cluster-tilting objects and potentials}, to appear
in Amer. J. Math., {arXiv:0804.3813}.


\bibitem{BuanMarshReiten07}
Aslak Bakke~Buan, Robert J. Marsh, and Idun Reiten,
   \emph{Cluster-tilted algebras},
   {Trans. Amer. Math. Soc.} \textbf{359} (2007), no.~{1}, {323--332},
   {electronic}.

\bibitem{BuanMarshVatne10}
Aslak Bakke~Buan, Robert J. Marsh, and Dagfinn F. Vatne,
   \emph{Cluster structures from 2-{C}alabi--{Y}au categories with loops},
   Math. Z. \textbf{265} (2010),  no.~4, 951--970.

\bibitem{BarotKussinLenzing08}
Micheal Barot, Dirk Kussin, and Helmut Lenzing,
     \emph{The {G}rothendieck group of a cluster category},
   {J. Pure Appl. Algebra}
  \textbf{212} (2008), no.~{1}, {33--46}.

\bibitem{BuanVatne08}
Aslak Bakke~Buan, and Dagfinn F. Vatne,
 \emph{Derived equivalence classification for cluster-tilted algebras
              of type {$A_n$}}, {J. Algebra} \textbf{319} (2008), no.~{7},
              {2723--2738}.

\bibitem{BurbanIyamaKellerReiten08}
Igor Burban, Osamu Iyama, Bernhard Keller, and Idun Reiten,
\emph{Cluster tilting for one-dimensional hypersurface
              singularities},
{Adv. Math.} \textbf{217} (2008), no.~6, {2443--2484}.


\bibitem{CalderoChapotonSchiffler06}
Philippe Caldero,  Fr\'ed\'eric Chapoton, and Ralf Schiffler,
\emph{Quivers with relations arising from clusters (${A}_n$ case)},
{Trans. Amer. Math. Soc.} \textbf{358} (2006), no.~{5},
{1347--1364}.


\bibitem{DerksenWeymanZelevinsky08}
Harm Derksen, Jerzy Weyman, and Andrei Zelevinsky, \emph{Quivers
with potentials and their representations {I}:
  {Mutations}}, Selecta Mathematica \textbf{14} (2008), 59--119.

\bibitem{FominZelevinsky02}
Sergey Fomin and Andrei Zelevinsky, \emph{Cluster algebras. {I}.
  {F}oundations}, J. Amer. Math. Soc. \textbf{15} (2002), no.~2, 497--529
  (electronic).

\bibitem{GeissReiten05}
Christof Gei{\ss}, and Idun Reiten, \emph{Gentle algebras are
{G}orenstein}, {Representations of algebras and related topics},
{Fields Inst. Commun.}, vol. \textbf{45}, {Amer. Math. Soc.},
{Providence, RI}, {2005}, pp. {129--133}.

\bibitem{HappelRingel} Dieter Happel, and Claus Michael Ringel,
\emph{Construction of tilted algebras}, Representations of algebras
(Puebla 1980), Lecture Notes in Math. \textbf{903}, Springer,
Berlin-New York, 1981, pp. 125--144.

\bibitem{HappelUnger89}
Dieter Happel, and Luise Unger, \emph{Almost complete tilting
modules}, {Proc. Amer. Math. Soc.} \textbf{107} (1989), no.~{3},
{603--610}.

\bibitem{HappelUnger05}
Dieter Happel, and Luise Unger, \emph{On a partial order of tilting
modules}, {Algebr. Represent. Theory} \textbf{8} (2005), no.~{2},
{147--156}.


\bibitem{Holm05}
Thorsten Holm, \emph{Cartan determinants for gentle algebras},
{Arch. Math. (Basel)} \textbf{85} (2005), no.~{3}, {233--239}.


\bibitem{HuXi08} Wei Hu and Changchang Xi,
    \emph{$\mathcal{D}$-split sequences and derived equivalences},
Adv. Math. \textbf{227} (2011), 292--318.

\bibitem{IyamaYoshino08}Osamu Iyama and Yuji Yoshino,
\emph{{Mutations in triangulated categories and rigid
        Cohen-Macaulay modules}}, {Inv. Math.} \textbf{172} (2008),
        {117--168}.


\bibitem{Keller08c}
Bernhard Keller, \emph{Cluster algebras, quiver representations and
  triangulated categories}, Triangulated
  categories (edited by Holm, J{\o}rgensen and Rouquier), London
  Math. Soc. Lecture Note Ser., vol. \textbf{375}, Cambridge University
  Press, Cambridge, 2010, pp. 76--160.


\bibitem{KellerQuiverMutationApplet}
\bysame, \emph{Quiver mutation in {J}ava}, {Java applet available at
the author's home page}.


\bibitem{Keller05}
\bysame, \emph{{On triangulated orbit categories}}, {Doc. Math.}
    \textbf{10} (2005), {551--581}.

\bibitem{KellerReiten07}
Bernhard Keller and Idun Reiten, \emph{{Cluster-tilted algebras are
Gorenstein
  and stably Calabi-Yau}}, Adv. Math. \textbf{211} (2007),
  123--151.

\bibitem{KoenigZhu08}
Steffen Koenig and Bin Zhu, \emph{From triangulated categories to
abelian categories: cluster tilting in a general framework}, {Math.
Z.} \textbf{258} (2008), no.~{1}, {143--160}.

\bibitem{Ladkani10}
Sefi Ladkani, \emph{Perverse equivalences, {BB}-tilting, mutations
and applications}, {arXiv:1001.4765v1}.

\bibitem{Plamondon10}
 Pierre-Guy Plamondon,
  \emph{Cluster characters for cluster categories with infinite-dimensional morphism spaces},
Adv. Math. \textbf{227} (2011), no.~1, 1--39.


\bibitem{Ringel07}
Claus~Michael Ringel, \emph{Some remarks concerning tilting modules
and tilted
  algebras. {Origin. Relevance. Future.}}, Handbook of Tilting Theory, London Math.
  Soc. Lecture Note Ser., vol. \textbf{332}, Cambridge Univ. Press, Cambridge, 2007,
  pp.~49--104.


\bibitem{Vatne11}
Dagfinn F. Vatne, \emph{Endomorphism rings of maximal rigid objects
in cluster tubes}, Colloq. Math. \textbf{123} (2011), 63--93.


\bibitem{ZhouZhu10}
{Yu Zhou and Bin Zhu}, \emph{Maximal rigid subcategories in
2-{C}alabi--{Y}au triangulated categories}, {arXiv:1004.5475}.

\bibitem{ZhouZhu10b} \bysame, \emph{Cluster algebras of
type $C$ via cluster tubes}, arXiv:1008.3444.

\end{thebibliography}

\def\cprime{$'$}
\providecommand{\bysame}{\leavevmode\hbox
to3em{\hrulefill}\thinspace}
\providecommand{\MR}{\relax\ifhmode\unskip\space\fi MR }
\providecommand{\MRhref}[2]{%
  \href{http://www.ams.org/mathscinet-getitem?mr=#1}{#2}
} \providecommand{\href}[2]{#2}

\end{document}